\documentclass[11pt]{amsart}

\usepackage{amsmath,amsfonts,amsthm,amssymb,graphicx}
\usepackage[all,2cell,ps,matrix,arrow,curve,frame]{xy}

\theoremstyle{plain}
\newtheorem{Theorem}{Theorem}[section]
\newtheorem{Lemma}[Theorem]{Lemma}

\newtheorem{Corollary}[Theorem]{Corollary}

\theoremstyle{definition}
\newtheorem{Remark}[Theorem]{Remark}

\newtheorem{Question}[Theorem]{Question}
\newtheorem{ConQuestions}[Theorem]{Concluding Questions}
\newtheorem{Example}[Theorem]{Example}

\newcommand{\grass}{\mathfrak{Grass}^T_d}
\newcommand{\ledeg}{\le_{\mathrm{deg}}}

\newcommand{\edge}{\ar@{-}}
\newcommand{\dttdar}{\ar@{.>}}

\newcommand{\dashedge}{\ar@{--}}
\newcommand{\dshdar}{\ar@{-->}}

\newcommand{\loopNE}{\ar@'{@+{[0,0]+(6,2)} @+{[0,0]+(10,10)}
@+{[0,0]+(2,6)}}}
\newcommand{\loopNW}{\ar@'{@+{[0,0]+(-2,6)} @+{[0,0]+(-10,10)}
@+{[0,0]+(-6,2)}}}
\newcommand{\loopSW}{\ar@'{@+{[0,0]+(-6,-2)} @+{[0,0]+(-10,-10)}
@+{[0,0]+(-2,-6)}}}
\newcommand{\loopSE}{\ar@'{@+{[0,0]+(2,-6)} @+{[0,0]+(10,-10)}
@+{[0,0]+(6,-2)}}}

\newcommand{\loopNNE}{\ar@'{@+{[0,0]+(4,2)} @+{[0,0]+(6,11)}
@+{[0,0]+(0,6)}}}
\newcommand{\loopSSW}{\ar@'{@+{[0,0]+(-4,-2)} @+{[0,0]+(-6,-12)}
@+{[0,0]+(0,-6)}}}
\newcommand{\loopSSE}{\ar@'{@+{[0,0]+(0,-6)} @+{[0,0]+(6,-11)}
@+{[0,0]+(4,-2)}}}

\begin{document}

\title{Orbit closures and rational surfaces}

\author{Frauke M. Bleher}
\address{F.B.: Department of Mathematics\\University of Iowa\\
Iowa City, IA 52242-1419}
\email{frauke-bleher@uiowa.edu}
\thanks{The first author was partially supported by NSA Grant H98230-11-1-0131.}
\author{Ted Chinburg}
\address{T.C.: Department of Mathematics\\University of Pennsylvania \\ 
Philadelphia, PA 19104-6395}
\email{ted@math.upenn.edu}
\thanks{The second author was partially supported by NSF grant DMS 1100355.} 
\author{Birge Huisgen-Zimmermann}
\address{B.H.: Department of Mathematics\\University of California Santa Barbara\\
Santa Barbara, CA 93106}
\email{birge@math.ucsb.edu}
\thanks{The third author was partially supported by NSF grant DMS 0500961
and NSF grant DMS 0932078, while in residence at MSRI, Berkeley.}

\begin{abstract}In this paper we study the Grassmannian of submodules of a given
dimension inside a finitely generated projective module $P$ for a finite dimensional algebra
$\Lambda$ over an algebraically closed field.  The orbit of such a submodule $C$ under the
action of $\mathrm{Aut}_\Lambda (  P  )$ on the Grassmannian encodes information on the degenerations of $P/C$
and has been considered by a number of authors.  The goal of this article is to bound the geometry of two-dimensional orbit closures in terms of representation-theoretic data. Several examples are given to illustrate the interplay between the geometry of the projective surfaces which arise and the corresponding posets of degenerations. 
  \end{abstract}

\maketitle


\section{Introduction}
\label{sec:intro}
\setcounter{equation}{0}

Let $\Lambda$ be a basic finite dimensional algebra over an algebraically closed field $k$.
A fundamental problem in representation theory is the classification of finitely generated
$\Lambda$-modules.  An early result in this direction is the Jordan-H\"older
Theorem, which groups together modules that have the same composition factors
counting multiplicities.  This is a very coarse classification, however, serving only as a basis for the study of  
finer groupings into collections of ``similar" modules.  Typically one starts with a rough subdivision of a given Jordan-H\"older class in terms of numerical invariants by, for instance, asking that certain of the simple composition factors hold prescribed positions   (such as placement in the radical quotient of the considered modules).
This leads to the study of 
partial orders on isomorphism classes of finitely generated $\Lambda$-modules $M$ and $N$, with $M \le N$ signifying that
$N$ results from $M$ by some form of simplification.  
For example, if we say $M \le N$ if and only
if $N$ is the semi-simplification of $M$, then we recover the Jordan H\"older class associated
to $N$ by this partial order.  
There are different partial orders under consideration (see, for example, \cite{bong1,bong2}, and
also \cite[Sect. 3]{birgeLMS} for a short overview).    

In this paper, we will concentrate on the geometrically defined degeneration partial order.  
Interest in this partial order arose from work of  
Gabriel and Kac on the affine scheme $\mathbf{Mod}_d(\Lambda)$ that parameterizes
the left $\Lambda$-modules with fixed dimension $d$ \cite{gabriel1,gabriel2,kac1,kac2}. 
The reductive group $\mathrm{GL}_d$ acts on $\mathbf{Mod}_d(\Lambda)$ by conjugation, and the
orbits under this action are in one-to-one correspondence with the isomorphism classes of 
$d$-dimensional $\Lambda$-modules. Suppose a left $\Lambda$-module $M$
corresponds to a point $x$ in the scheme $\mathbf{Mod}_d(\Lambda)$. A degeneration of $M$
is any $\Lambda$-module $N$ corresponding to a point in the closure of
the orbit of $x$ under $\mathrm{GL}_d$.  By setting $M \ledeg N$ in this situation, we arrive at a partial order, which places increasingly simplified modules $N$ above $M$, the largest being the direct sum of the composition factors of $M$.   
This partial order was studied by Kraft, Riedtmann, Bongartz, Schofield,
Skowronski, Zwara and many others  (see for example \cite{kraft1,kraft2}, 
\cite{riedt,riedtscho}, \cite{bong1,bong2}, \cite{skowzwara,zwara1,zwara2}).  One of the highlights is a purely algebraic description of the degeneration order in terms of Riedtmann-Zwara exact sequences in \cite{zwara2}.

Let $J$ be the Jacobson radical of $\Lambda$,
and suppose $T$ is a finitely generated semisimple $\Lambda$-module.  In this paper we will
consider $\Lambda$-modules $M$ that have radical quotient $M/JM$ isomorphic to $T$.
There is a $\mathrm{GL}_d$-stable locally closed subscheme $\mathbf{Mod}^T_d$  of $\mathbf{Mod}_d(\Lambda)$ 
whose orbits are in bijection with the isomorphism classes of such modules $M$.

In \cite{bonghuisgen1,bonghuisgen2}, \cite{birgeTAMS,birgeLMS,birgehier}, Bongartz and
the third author
took an alternate geometric approach to studying degenerations $M \ledeg N$. When $M$ and $N$
have the same radical quotient $T$ this approach proceeds in the following way.  Let $P$ be a projective $\Lambda$-module with radical quotient
$T$.  Then $M \cong P/C$ for some $\Lambda$-submodule $C$ of $JP$.   We consider $C$ as a point in the projective scheme $\grass$ 
of all submodules of $JP$ that have codimension $d$ in $P$, and we let the group 
$\mathrm{Aut}_\Lambda(  P  )$ act canonically on $\grass$.  As above, orbits correspond bijectively to isomorphism classes of modules $N$ with radical quotient $T$. The relation $M \ledeg N$  is equivalent to the existence of a point $C'$ in the orbit closure
$\overline{\mathrm{Aut}_\Lambda(  P  ) . C}$ such that $N \cong P/C'$.  A significant advantage of this
approach is that orbit closures become closed subsets of projective varieties rather than of affine ones.

Our primary goal in this paper is to bound the global geometry of $\overline{\mathrm{Aut}_\Lambda(  P  ) . C}$ in terms of the
representation theoretic data which specifies $M = P/C$.  
In particular, we will study the following problem:
\begin{Question}
\label{qu:Euler}
Is the Euler characteristic of $\overline{\mathrm{Aut}_\Lambda(  P  ) . C}$ bounded
by a function of the dimension of $\mathrm{Aut}_\Lambda(  P  ) . C$, the field $k$ 
and the dimension of $C$ over $k$?
\end{Question}

We answer Question \ref{qu:Euler} in the affirmative for simple $T$ and orbits of dimension $2$.  In fact, using the classical theory of rational surfaces, we prove the following (see also Theorem \ref{thm:main}):
\begin{Theorem}

\label{thm:mainintro}
Suppose $T$ is simple and $\mathrm{Aut}_\Lambda(  P  ) . C$
has dimension $2$. Then there exist positive integers $b$ and $c$ depending only on $k$ and $\mathrm{dim}_k( C )$, together with a collection of $b$ relatively minimal smooth rational projective surfaces, such that
the minimal desingularization of $\overline{\mathrm{Aut}_\Lambda(  P  ) . C}$ 
can be obtained from one of the surfaces in this collection by 
performing at most $c$ monoidal transformations.   
\end{Theorem}

\begin{Corollary}
\label{cor:Euler}
Under the hypotheses of Theorem $\ref{thm:mainintro}$, the Euler characteristic of
$\overline{\mathrm{Aut}_\Lambda(  P  ) . C}$ is bounded from above by a function of $k$
and $\mathrm{dim}_k( C )$.
\end{Corollary}

Without
assuming that the orbit $\mathrm{Aut}_\Lambda( P ) . C $ is two-dimensional, we will
in fact bound, in a similar way, the geometry of the closure of any two-dimensional affine plane  
contained in  $\mathrm{Aut}_\Lambda(  P  ) . C$ (see Remark \ref{rem:moregeneral}).
We should note that  if $T$ is simple, the orbit $\mathrm{Aut}_\Lambda(  P  ) . C$ is an affine space (see \cite[Prop. 2.9]{birgeTAMS} and \cite[Lem. 4.1]{birgeLMS}). 

In the final section, we give numerous examples in which we explicitly link the geometry of $\overline{\mathrm{Aut}_\Lambda (  P  ).C}$ to the top-stable degenerations of $P/C$. 

We now give a brief overview of existing work on the geometry of orbit closures.

 In \cite{bongsing}, Bongartz analyzed the singularities at minimal degenerations for modules over 
representation-finite hereditary algebras.
In \cite{bendbong}, Bender and Bongartz considered the case when $\Lambda$ is
the Kronecker algebra. 
They
classified all minimal singularities up to smooth equivalence and showed 
that they are isolated Cohen-Macaulay.

In \cite{zwarasmooth}, Zwara proved that if $X$ is an orbit closure in $\mathbf{Mod}_d(\Lambda)$ and $y\in X$ is an element 
such that the orbit of $y$ has codimension 1 in $X$, then $X$ is smooth at $y$.
In \cite{zwarasmooth2}, he moreover showed that if $\Lambda$ is the path algebra of a Dynkin quiver, then
the orbit closure $X$ is regular in codimension two.

\section{Preliminaries: conventions and basic results}
\label{sec:prelim}
\setcounter{equation}{0}

In this section, we set up our notation and summarize what we need regarding the classification of relatively minimal smooth rational projective surfaces.
Let $k$ be an algebraically closed field of arbitrary characteristic, and let 
$\Lambda$ be a basic finite dimensional $k$-algebra with Jacobson radical $J$.
Without loss of generality, we assume that $\Lambda=kQ/I$ for some finite quiver $Q$ and some admissible
ideal $I$ of the path algebra $kQ$. The quiver $Q$ provides us with a distinguished set of primitive
idempotents $e_1,\ldots,e_n$ of $\Lambda$, which are in bijective correspondence with the vertices
of $Q$. As is well-known, the quotient modules $\Lambda e_i/ J e_i$, $1\le i\le n$, form a complete set
of representatives for the isomorphism classes of simple $\Lambda$-modules. 

We fix a simple $\Lambda$-module $T$, corresponding to a primitive idempotent $e=e_{i_0}$ for some $i_0$,
together with its projective cover $P=\Lambda e$.
Let $d$ be a positive integer with $1 < d < \mathrm{dim}_k(JP)$, and
let $d'=\mathrm{dim}_k(P)-d$. Denote the classical Grassmannian of $d'$-dimensional subspaces of the
$k$-vector space $JP$ by $\mathcal{G}r(d',JP)$. We define 
$$\grass = \{ C\in \mathcal{G}r(d',JP)\;|\; C\mbox{ is a $\Lambda$-submodule of }JP\},$$ 
which is a closed subscheme of $\mathcal{G}r(d',JP)$.
We have an obvious surjection $\phi$ from $\grass$ to the set of isomorphism classes of $d$-dimensional 
$\Lambda$-modules with radical quotient $T$, where $\phi(C)=[P/C]$. The fibers of $\phi$ coincide
with the orbits of the natural action of $\mathrm{Aut}_\Lambda(P)$ on $\grass$.

Suppose $C, C' \in \grass$. As we mentioned in the introduction, the partial order 
$M \ledeg N$ on $\Lambda$-modules of the form $M = P/C$
and $N = P/C'$, defined by requiring that $C'$ be in the closure of the $\mathrm{Aut}_\Lambda(P)$-orbit of $C$, coincides with the degeneration order based on the $\mathrm{GL}_d$-action on $\mathbf{Mod}^T_d$.  
In fact, in \cite[Prop. C]{bonghuisgen2}, it is shown that there is an inclusion-preserving
bijection between the $\mathrm{Aut}_{\Lambda}(P)$-stable subsets of $\grass$ and
the $\mathrm{GL}_d$-stable subsets of $\mathbf{Mod}^T_d$ which preserves and reflects openness, closures,
connectedness, irreducibility, and types of singularities.  Hence geometric results  concerning orbit closures in $\grass$ can, to a large extent, be carried over to orbit closures in $\mathbf{Mod}^T_d$, and vice versa.

Since we assume $T$ to be simple, \cite[Prop. 2.9]{birgeTAMS} shows the orbit
$\mathrm{Aut}_\Lambda(P).C$ to be isomorphic to a full affine space $\mathbb{A}^{\frak{m}}_k$, where
$$\frak{m}=\mathrm{dim}_k\,\mathrm{Hom}_\Lambda(P,JP/C)-\mathrm{dim}_k\,\mathrm{Hom}_\Lambda(
P/C,JP/C).$$
More precisely,
let $\{\omega_1,\ldots,\omega_\mu\}$ be a $k$-basis of  $eJe$ consisting of oriented cycles from $e$ to $e$, and
$\mathrm{Stab}_{eJe}(C)$ the $k$-vector space consisting of the elements $a\in eJe$ for which 
$Ca\subseteq C$.  Suppose $\{\omega_1,\ldots,\omega_\frak{m}\}$ is a $k$-basis of $eJe$ modulo
$\mathrm{Stab}_{eJe}(C)$. The following is shown in Lemma 4.1 of \cite{birgeLMS}:

\begin{Lemma}
\label{lem:birgeaffine}
Let $\overline{\mathrm{Aut}_\Lambda(P) . C}$ be the closure of $\mathrm{Aut}_\Lambda(P) . C$
in $\grass$ with the induced reduced structure.  There is a morphism 
$$\Psi:\mathbb{A}^{\frak{m}}_k \to \overline{\mathrm{Aut}_\Lambda(P) . C}$$ 
which is an isomorphism from $\mathbb{A}^{\frak{m}}_k$ to the dense open subset $\mathrm{Aut}_\Lambda(P) . C$ of
 the target.  It sends the point with coordinates $(t_1,\ldots,t_{\frak{m}})$ in  
$\mathbb{A}^{\frak{m}}_k$ to the element $C\cdot (e + t_1 \omega_1 + \cdots + t_{\frak{m}} \omega_{\frak{m}})$ of  
$\mathrm{Aut}_\Lambda(P) . C$.
\end{Lemma}

Orbit closures in both the affine and projective module schemes are always unirational, since both the general linear group and the automorphism group $\mathrm{Aut}_\Lambda(P)$ are connected rational, and the orbits are epimorphic images of these groups, respectively.  There is no example where rationality has been found to fail, and in our present situation, rationality is actually guaranteed, as it is in the more general case when $T$ has no simple summands of multiplicity $>1$ (see \cite[Theorem 5.1]{birgeLMS}).

\begin{Corollary}
\label{cor:birgeaffine}
The orbit closure  $\overline{\mathrm{Aut}_\Lambda(P) . C}$
in $\grass$ is a rational variety of dimension $\frak{m}$.
\end{Corollary}

Unless specified otherwise, we now assume that $\frak{m} =2$.   Then $\mathrm{Aut}_\Lambda(P) . C$ is isomorphic to $\mathbb{A}^2_k$, and hence
$\overline{\mathrm{Aut}_\Lambda(P) . C}$ is a rational projective surface. 
We use the following notation:
\begin{itemize}
\item $\overline{\mathrm{Aut}_\Lambda(P) . C}^\#$ denotes the normalization of $\overline{\mathrm{Aut}_\Lambda(P) . C}$, and
\item $\overline{\mathrm{Aut}_\Lambda(P) . C}^\dagger$ denotes the 
minimal desingularization of $\overline{\mathrm{Aut}_\Lambda(P) . C}^\#$.
\end{itemize}
In general, the smooth surface $\overline{\mathrm{Aut}_\Lambda(P) . C}^\dagger$  
over $k$ fails to be relatively minimal.
Recall that a smooth irreducible projective surface $X$ over $k$
is said to be relatively minimal if every birational morphism from $X$ to another
smooth projective surface  is necessarily an isomorphism. The relatively minimal smooth 
rational projective surfaces are known to be, up to isomorphism, the surfaces on the following list:
$$ \mathbb{P}_k^2\, ;\quad  X_0=\mathbb{P}_k^1 \times \mathbb{P}_k^1\,;\quad X_n\,, \ n\ge 2$$
(see \cite[Example V.5.8.2, Rem. V.5.8.4]{Hartshorne}).  The $X_n$ are referred to as the Hirzebruch surfaces and  are described in \cite[Sect. V.2]{Hartshorne}. Each $X_n$ is a rational ruled projective 
surface defined by
$$X_n = \mathbf{P}\left(\mathcal{E}\right) \quad\mbox{for}\quad
\mathcal{E}=\mathcal{O}_{\mathbb{P}^1_k}\oplus\mathcal{O}_{\mathbb{P}^1_k}(-n) \quad\mbox{over}\quad \mathbb{P}^1_k.$$
Here, $\mathbf{P}(\mathcal{E})=\mathbf{Proj}(\mathrm{Sym}(\mathcal{E}))$, where 
$\mathrm{Sym}(\mathcal{E})$
is the symmetric algebra of $\mathcal{E}$. It follows from 
\cite[Example V.5.7.1 and Exercise V.5.5]{Hartshorne} that each $X_n$ can be obtained from 
$X_0=\mathbb{P}^1_k\times \mathbb{P}^1_k$ by a finite sequence of monoidal transformations, via the following recursion:
Let $D_0=\mathbb{P}^1_k\times\infty$ and $D'_0=\infty\times\mathbb{P}^1_k$. 
For any nonnegative integer $i$, the following two steps (a) and (b) transform $X_i$ into $X_{i+1}$.   Suppose that $X_i$, together with curves $D_i$, $D'_i$ $\subset X_i$, has been determined.  

\begin{enumerate}
\item[(a)] Blow up the crossing point of $D_i$ and $D'_i$.  This results in a new exceptional curve $E_i$.
\item[(b)] Blow down the proper transform of $D_i$.  This leads to new boundary
curves $\widetilde{E_i}$ and $\widetilde{D'_i}$. Let $D_{i+1}=\widetilde{E_i}$ and $D'_{i+1}=\widetilde{D'_i}$.
\end{enumerate}
We thus obtain the following picture relating $\overline{\mathrm{Aut}_\Lambda(P) . C}$ to $\overline{\mathrm{Aut}_\Lambda(P) . C}^{\#}$ and
$\overline{\mathrm{Aut}_\Lambda(P) . C}^\dagger$:
\begin{equation}
\xymatrix @R=1pc{&\overline{\mathrm{Aut}_\Lambda(P) . C}^\dagger\ar[rd]\ar[ldd]_{\rho}\\
&&\overline{\mathrm{Aut}_\Lambda(P) . C}^{\#}\ar[rd]\\
X&&&\overline{\mathrm{Aut}_\Lambda(P) . C}  }
\end{equation}
Here all arrows stand for birational morphisms, and $\overline{\mathrm{Aut}_\Lambda(P) . C}^\#\to \overline{\mathrm{Aut}_\Lambda(P) . C}$ is finite. Moreover, $X$ is a relatively minimal smooth rational projective surface, and $\rho$ is a finite sequence of monoidal transformations.

\section{Orbit closures and rational surfaces}
\label{sec:surfaces}
\setcounter{equation}{0}

The formulation of our theorem refers to the notation introduced in section \ref{sec:prelim}. 
In particular, $T$ is a simple $\Lambda$-module with projective cover $P=\Lambda e$,  and $C$ is a point in $\grass$ for some fixed positive integer $d$.  Moreover, we assume  that the orbit
$\mathrm{Aut}_\Lambda(P) . C$ of $C$ in $\grass$ is two-dimensional, and denote by $\overline{\mathrm{Aut}_\Lambda(P) . C}$ its closure 
in $\grass$, the closure being endowed with the induced reduced structure.  

Recall that, by Lemma \ref{lem:birgeaffine}, there is an injective morphism 
\begin{equation}
\label{eq:psi}
\Psi:\mathbb{A}^2_k \to \overline{\mathrm{Aut}_\Lambda(P) . C}
\end{equation}
which maps $\mathbb{A}^2_k$ onto the open dense subset $\mathrm{Aut}_\Lambda(P) . C$ of
$\overline{\mathrm{Aut}_\Lambda(P) . C}$. 

The following is our main result.

\begin{Theorem}
\label{thm:main} 
For some natural number $n_0$ which depends only on $k$ and $\mathrm{dim}_k(C)$, there exists a birational morphism 
$$\rho:\overline{\mathrm{Aut}_\Lambda(P) . C}^\dagger \to  X$$ with the following properties:  $X$ is a relatively minimal smooth rational projective surface among
$$\mathbb{P}^2_k,\  \ X_n, \quad\mbox{for}\quad 0 \le n \le n_0, \ n \ne 1,$$
and there is a bound $c$, again depending solely on $k$ and $\mathrm{dim}_k(C)$, such that $\rho$ blows down at most $c$ irreducible curves.
\end{Theorem}

For simplicity we regard $k$ as fixed from now on, so that we will not have to address the
dependence on $k$ of various bounds.  When we say that a real-valued function is bounded, we mean that it is 
bounded from above by some explicit function of $\mathrm{dim}_k(C)$.

We prove Theorem \ref{thm:main} in several steps which are carried out in detail below.  In Step 1
we bound the number of points at which the birational map 
$\xymatrix{\psi:\mathbb{P}^1_k \times \mathbb{P}^1_k \ar@{-->}[r]& \overline{\mathrm{Aut}_\Lambda(P) . C}}$
resulting from  the morphism $\Psi$ in (\ref{eq:psi}) is not defined.  At each point $w$ where this rational map is undefined, we also
bound the complexity of the map $\psi$ in a neighborhood of $w$.  In Step 2 we show that for each such $w$, the birational map
$\psi$ becomes a rational morphism near $w$ after the blow-up of an ideal that contains a bounded power (i.e. a power with bounded exponent) of the maximal ideal of the local ring of $w$. In Step 3 we consider the Grassmannian consisting of all ideals of the local ring of $w$
which contain a bounded power of the maximal ideal.  We show that, to prove Theorem \ref{thm:main},
it is enough to prove that 
the blow-up of $\mathbb{P}^1_k \times \mathbb{P}^1_k$ at any ideal
in this Grassmannian can be dominated by a blow-up of $\mathbb{P}^1_k \times \mathbb{P}^1_k$ which results from a bounded number of 
successive monoidal transformations.
In Steps 4 and 5 we complete the proof using  a theorem
of Zariski  about dominating proper birational morphisms between normal projective surfaces using
a finite number of monoidal transformations.

\vskip .1in

\noindent \textbf{Step 1:} \textit{  Write $\mathbb{P}^1_k = \mathbb{A}^1_k \cup \{\infty\}$ and identify $\mathbb{A}^2_k $ with $\mathbb{A}^1_k \times \mathbb{A}^1_k$.  This gives an embedding of $\mathbb{A}^2_k$  into $\mathbb{P}^1_k \times \mathbb{P}^1_k$.
The morphism $\Psi$ from $(\ref{eq:psi})$
defines a birational map 
$$\xymatrix{\psi:\mathbb{P}^1_k \times \mathbb{P}^1_k \ar@{-->}[r]& \overline{\mathrm{Aut}_\Lambda(P) . C}}.$$
Let $U$ be the domain of definition of $\psi$.  Then the set $D = \mathbb{P}^1_k \times \mathbb{P}^1_k  - U$ of fundamental points of $\psi$ is a finite set of closed points, and there is an effective bound, depending only on $\mathrm{dim}_k(C)$,
for the cardinality of this set.  Moreover, there is a projective embedding of $\grass$
into a projective space $\mathbb{P}^h_k$ over $k$ satisfying the following conditions for each $w \in D$.  There are  local parameters $t_1$ and $t_2$ at $w$ such that the local ring 
of $w$ on $\mathbb{P}^1_k \times \mathbb{P}^1_k $ is isomorphic to the localization $A_w$ of
$k[t_1,t_2]$ at the maximal ideal generated by $t_1$ and $t_2$.  The restriction of $\psi$
to the open subset $\mathrm{Spec}(A_w) - w$ of $\mathrm{Spec}(A_w)$ is a morphism to $\mathbb{P}^h_k$
which is defined, in terms of homogeneous coordinates, by $(q_0(t_1,t_2):\cdots:q_h(t_1,t_2))$ where
the $q_i(t_1,t_2)$ are polynomials in $k[t_1,t_2]$ of bounded degree and the $A_w$-ideal $I$
generated by the $q_i(t_1,t_2)$ contains a bounded power of the maximal ideal of $A_w$.
(Note: We do not claim that $h$ is bounded.)}
\medbreak

\noindent {\bf Proof of Step 1:} 
Let $m = \mathrm{dim}_k(P)$ and fix a basis $\{b_i\}_{i = 1}^m$ for $P$ over $k$.  Relative to this
basis, $C$ is spanned by a set of $d' = \mathrm{dim}_k(C)$ row vectors in $k^m$.  Recall that the map
$\Psi:\mathbb{A}^2_k = \mathbb{A}^1_k \times \mathbb{A}^1_k \to \overline{\mathrm{Aut}_\Lambda(P) . C}$ of (\ref{eq:psi})
was constructed in the following way.  If $z, u \in k$ define a point $(z,u) \in \mathbb{A}^2_k$,
then
$$\Psi(z,u) = C \cdot (e + z\, \omega_1 + u \,\omega_2)$$
for some fixed elements $\omega_1$ and $\omega_2$ of $eJe$.  This implies that there is a set of  ${d'}$ row vectors $v_1,\cdots,v_{d'}$ of size $m$ whose entries are polynomials
which are at most linear in the indeterminates $z$ and $u$ such that, if one specializes $z$ and $u$ to elements of $k$,
the vectors $v_1,\ldots,v_{d'}$ specialize to independent vectors which span the subspace of $P$
corresponding to $\Psi(z,u) $. 

It is well-known that the set $D$ of fundamental points of $\psi$ is closed in $\mathbb{P}^1_k \times \mathbb{P}^1_k$ \,  and has codimension $2$ (see \cite[Lemma V.5.1]{Hartshorne}).  In other words,  the morphism $\Psi$ of (\ref{eq:psi}) has an extension which is defined
off a finite set of closed points.  We need to make this extension more explicit in order to
bound the number of fundamental points of $\psi$.

Let us first constructively extend $\Psi$ to a large open Zariski neighborhood of
the generic point of $\infty \times \mathbb{P}^1_k$.  The local ring of this generic point may
be identified with the localization $R$ of the polynomial ring $k[z^{-1},u]$ at the prime
ideal $k[z^{-1},u] \cdot z^{-1}$.  We will apply the following lemma to the situation where
$r = z^{-1}$,\, $s = u$, and $V$ is the matrix whose $(i,j)$ entry $v_{i,j}$
is the $j$-th component of the vector $v_i$.  

\begin{Lemma}
\label{lem:fixit}  Suppose $k[r,s]$ is a polynomial ring in two indeterminates $r$ and $s$.
Let $F = \mathrm{Frac}(k[r,s]) = k(r,s)$, and define $R$ to be the localization of $k[r,s]$ at the
prime ideal \hbox{$k[r,s]\cdot r$}.  Suppose $V = (v_{i,j})_{1 \le i \le {d'}, 1 \le j \le m}$ is a matrix
of elements of $F$ whose entries can each be written in the form $f(r,s)/\tilde{f}(r,s)$
such that $f(r,s), \tilde{f}(r,s) \in k[r,s]$ have bounded degrees.  Suppose
furthermore that the rows of $V$ are linearly independent over $F$, so that, over $F$,
they span a ${d'}$-dimensional subspace $W$ of $F^m$.  Viewing $R^m$ as canonically embedded into $F^m$, we obtain: 
\begin{enumerate}
\item[i.] $R^m \cap W$ is a free $R$-module direct summand of $R^m$ that has rank $d' = \mathrm{dim}_F W$.  In particular, $R^m \cap W$ contains a basis
for $W$ over $F$.   
\item[ii.]  There is a ${d'} \times {d'}$ matrix $Y = (y_{i,j})$ over $F$ 
with the following properties:  The  rows of $Y \cdot V$ form an $R$-basis for $R^m \cap W$.  Each $y_{i,j}$
can be written as a ratio $g(r,s)/\tilde{g}(r,s)$ in which $g(r,s), \tilde{g}(r,s) \in k[r,s]$ are polynomials
of bounded degrees, the bound depending only on ${d'} = \mathrm{dim}_k(C)$ and the given bound on the degrees of the numerators and denominators of the $v_{i,j}$.
\end{enumerate}
\end{Lemma}

\begin{proof} The ring $R$ is a discrete valuation ring, and $R^m \cap W$ is a pure $R$-submodule
of $R^m$ in the sense that $R^m / (R^m \cap W)$ is $R$-torsion free. Thus $R^m / (R^m \cap W)$ is free over $R$, meaning that $R^m \cap W$ is a free direct summand of $R^m$.  In particular, $R^m \cap W$ is free of rank ${d'} = \mathrm{dim}_F W$ as an $R$-module.

Let $\mathrm{ord}:F \to \mathbb{Z}$ be the discrete valuation on $F = k(r,s)$
which is associated to the prime ideal $k[r,s] \cdot r$ of $k[r,s]$.  Our assumptions on the $v_{i,j}$
imply that there is an integer $l \ge 0$ such that $v_{i,j} \ne 0$ implies 
$|\mathrm{ord}(v_{i,j})| \le l$.

We prove the rest of Lemma \ref{lem:fixit} using 
induction on ${d'}$.  If ${d'} = 1$, then $V$ has a single row, and $W$ is the $F$-space
spanned by this row.  If $t = \mathrm{min}_{j = 1}^m \mathrm{ord}(v_{1,j})$, then
$R^m \cap W$ is the free $R$-module on the vector $r^{-t}(v_{1,1},\ldots,v_{1,m})$.
So  we can take the $1 \times 1$ matrix $Y$ to be $(r^{-t})$.  Since some $v_{1,j}$ is non-zero, we have $|t| \le l$.  

Suppose now that ${d'} > 1$.  Let $(i,j)$ be a pair of indices with $1 \le i \le {d'}$
and $1 \le j \le m$ such that $q = \mathrm{ord}(v_{i,j})$ is minimal among the orders of
the non-zero entries of $V$. As above, $|q| \le l$. By multiplying
$V$ by a ${d'} \times {d'}$ permutation matrix on the left, we may assume $i = 1$.
The vector  $r^{-q}(v_{1,1},\ldots,v_{1,m})$
lies in $R^m \cap W$ and $r^{-q} v_{1,j}$ is a unit of $R$.  We multiply 
$V$ by the diagonal matrix $Y_0 = \mathrm{diag}(r^{-q},1,\ldots,1)$
to make the first row of $Y_0 V$
equal to $r^{-q}(v_{1,1},\ldots,v_{1,m})$.  We now subtract $v_{a,j}/(r^{-q} v_{1,j})$ times
the first row of $Y_0 V$ from the $a^{th}$ row of $Y_0 V$ for $2 \le a \le {d'}$
to arrive at a matrix $Y_1 Y_0 V$ which has first row $r^{-q}(v_{1,1},\ldots,v_{1,m})$,
zero entries in the $j$-th column except for the unit $r^{-q} v_{1,j}$ such that the $F$-span of the rows of $Y_1 Y_0 V$ equals $W$.
Let $W'$ be the $F$-span of rows $2,\ldots,{d'}$ of $Y_1 Y_0 V$.  Note
that the $j$-th component of every element of $W'$ is zero.  We claim that
\begin{equation}
\label{eq:first}
R^m \cap W = (R\cdot r^{-q}(v_{1,1},\ldots,v_{1,m})) \oplus (R^m \cap W').
\end{equation}
It is clear that the right hand side is contained in the left hand side, since $W$
is the $F$-span of all the rows of $Y_1 Y_0 V$.  For the opposite containment,
suppose that some $F$-linear combination of the rows of $Y_1 Y_0 V$
lies in $R^m$.  If $\alpha \in F$ is the coefficient of the first row in this linear
combination, then $\alpha \cdot r^{-q} v_{1,j}$ is the $j^{th}$ component
of the linear combination, and this must be in $R$.  Since $r^{-q} v_{1,j}$
is a unit in $R$, this forces $\alpha \in R$.  So the multiple of the first
row in the linear combination lies in $R\cdot r^{-q}(v_{1,1},\ldots,v_{1,m})$, and
on subtracting this off we get an element of $R^m \cap W'$. This
proves (\ref{eq:first}). 

Because of our assumptions about the degrees of the numerators and denominators
of the non-zero $v_{i',j'}$, the non-zero entries of $Y_1 Y_0$  and of $Y_1 Y_0 V$ have numerators and denominators
of bounded degree.  We now apply our induction hypotheses to the $({d'} -1 ) \times m$ matrix $V'$ 
whose rows and columns are those of $Y_1 Y_0 V$ when we omit the first row. 
 This leads to a $({d'}-1) \times ({d'}-1)$ matrix $Y'$ such that the rows
of $Y' V'$ form a basis for $R^m \cap W'$.  We define $Y_2$ to be the ${d'} \times {d'} $ block
matrix with a one-by-one block equal to $1$ in the upper left corner followed
by a $({d'}-1) \times ({d'}-1)$ block given by $Y'$.  Now $Y_2 Y_1 Y_0 V$
has first row $r^{-q}(v_{1,1},\ldots,v_{1,m})$ and the remaining rows
span the free $R$-module $R^m \cap W'$.  Because of (\ref{eq:first}),
we may assume $Y = Y_2 Y_1 Y_0$.  The degrees of the numerators and
denominators of the non-zero entries of $Y$ are bounded since this is true for the non-zero entries of $Y_2$, $Y_1$ and $Y_0$.
\end{proof}  

We now plug into Lemma \ref{lem:fixit} the substitutions mentioned in the paragraph
just before the statement of Lemma \ref{lem:fixit}. In particular, $r = z^{-1}$ and $s = u$, and the
entries of $rV$ are polynomials in $r$ and $u$. 
We write the rows of $V$ as $v_1, \dots, v_{d'}$.  
Given $Y$ as guaranteed by the lemma, we write $Y\cdot V = (q_{i,j})_{1 \le i \le {d'}, 1 \le j \le m}$, where $q_{i,j} = q_{i,j}(z^{-1},u)$
is a ratio of polynomials of bounded degrees in the indeterminates $r = z^{-1}$
and $u$. The fact that $q_{i,j}$ lies in the discrete valuation ring $R = k[z^{-1},u]_{(z^{-1})}$
implies that $z^{-1}$ does not divide the denominator of any non-zero $q_{i,j}$
when $q_{i,j}$ is written as a quotient of coprime polynomials in $k[z^{-1},u]$.
By construction, the rows $q_i = (q_{i,1},\ldots,q_{i,m})$ of $Y \cdot V$
span the $R$-module $R^m \cap W$. We will need the following additional
fact.  Let $y$ be a least common multiple in $k[z^{-1},u]$ of the denominators 
of the entries of $Y$, so that $y$ is well defined up to multiplication by an
element of $k^*$.  Let $y'$ be the quotient of $y$ by the highest power
of $z^{-1}$ which divides $y$.  Due to $q_{i,j} \in R$, we conclude that the entries of $y' \,Y \cdot V$ are polynomials
in $k[z^{-1},u]$.  Furthermore, since the entries of $Y$ have numerators and denominators of bounded degree, the
degree of $y'$ is bounded as well. 

We have already shown that 
$R^m \cap W$ is a free $R$-module summand
of $R^m$.  This implies that the image of $R^m \cap W$ in $(R/Rz^{-1})^m$
has dimension ${d'}$ over the field $R/Rz^{-1} \cong k(u)$.  Therefore, we
can use the  rows of $Y \cdot V$ to define a map from a non-empty open subset of
$\mathrm{Spec}(k[z^{-1},u])$ to the Grassmannian  $\grass$
which contains a dense open subset of the affine line $\infty \times \mathbb{A}^1_k$
defined by setting $z^{-1}$ equal to $0$.  Since $Y$ is generically invertible,
this map extends the morphism $\Psi : \mathbb{A}^1_k \times \mathbb{A}^1_k \rightarrow \grass$.  To see at which of the points in $\infty \times \mathbb{A}^1_k$ this extension fails to be defined,
let $\overline{(Y\cdot V)}$ be the ${d'} \times m$ matrix with entries
in $R/Rz^{-1} = k(u)$ which results from reducing the entries of $Y \cdot V$
modulo $R z^{-1}$.  Since the entries of $Y \cdot V$ are ratios of polynomials
of bounded degree in $z^{-1}$ and $u$ and, when written in terms of coprime numerators and denominators, none of these ratios has a denominator divided by $z^{-1}$, we conclude that the entries of  $\overline{(Y\cdot V)}$
are ratios of polynomials in $k[u]$ of bounded degrees.  The rows of
$\overline{(Y\cdot V)}$ are linearly independent over $k(u)$, since
the image of $R^m \cap W$ in $(R/Rz^{-1})^m = k(u)^m$
has dimension ${d'}$ over $k(u)$.  We showed
above that there is a polynomial $y' \in k[z^{-1},u]$ of bounded degree which is not divisible
by $z^{-1}$ such that $y' \,Y \cdot V$ has all its entries in $k[z^{-1},u]$.  It follows that if $\overline{y'} \in k[u]$
is the reduction of $y'$ mod $Rz^{-1}$, then $\overline{y'} \ne 0$ and the denominator
of every element of $\overline{(Y\cdot V)}$ divides $\overline{y'}$.  Thus the rows
of $\overline{(Y\cdot V)}$ can be specialized to every point of $\infty \times \mathbb{A}^1_k$
which is not a zero of the polynomial $\overline{y'} \in k[u]$, and the number of such
points is bounded.  We need to show that these rows are independent off a bounded
set of points of $\infty \times \mathbb{A}^1_k$.  This is so because there is a 
${d'} \times {d'}$ minor
of $\overline{(Y\cdot V)}$ whose determinant is not $0$ in $k(u)$.
This determinant is also a ratio of polynomials in $k[u]$ of bounded
degree, so the number of zeros and poles of the determinant is bounded.
So we have now bounded the number of fundamental points on $\infty \times \mathbb{A}^1_k$
 of the rational map 
\begin{equation}
\label{eq:psiG}
\xymatrix{\psi:\mathbb{P}^1_k \times \mathbb{P}^1_k \ar@{-->}[r]&  \overline{\mathrm{Aut}_\Lambda(P) . C} \hookrightarrow \grass. }
\end{equation}
One can similarly bound the number of points of $\mathbb{A}^1_k \times \infty$
which lie outside the domain of definition of $\psi$.  Since there is only 
one other point $\infty \times \infty$ where the rational map might be undefined, this effectively bounds the total number of fundamental points of $\psi$.

We now show the remaining claims of Step 1, concerning
the restriction of $\psi$ to $\mathrm{Spec}(A_w) - w$ where $A_w$ is the local
ring of $\mathbb{P}^1_k \times \mathbb{P}^1_k$ at a fundamental point $w$ of $\psi$.
To simplify notation, we set $A=A_w$.

Set $h = \left ( {m}\atop{d'} \right ) - 1 =  \left ( {m}\atop{d} \right ) - 1$.   We will find that 
the canonical embedding
$$\grass \hookrightarrow \mathcal{G}r(d',k^m) \overset{\iota} \to \mathbb{P}\left(\Lambda^{d'}(k^m)\right) = \mathbb{P}^h_k,$$
where $\iota$ sends any subspace $Z \subset k^m$ of dimension $d'$ to the point in the projective space defined by the $d'^{th}$ exterior power of a basis for $Z$, satisfies the requirements spelled out in Step 1.  Indeed, $\iota$ is well-known to be an injection with closed image, and the variety structure on $\mathcal{G}r(d',k^m)$ is defined so as to make $\iota$ an isomorphism from $\mathcal{G}r(d',k^m)$ onto a closed subvariety of  the projective space $\mathbb{P}\left(\Lambda^{d'}(k^m)\right)$, the latter endowed with its induced reduced structure.  Consequently, the restriction to $\grass$ is a closed immersion as well.  We keep this immersion fixed in the sequel and regard
$\psi$ as a rational map from $\mathbb{P}^1_k \times \mathbb{P}^1_k$
to $\mathbb{P}^h_k$.  

Let $w$ belong to the set $D$ of fundamental points of $\psi$, and let $A=A_w$ be the local ring of 
$\mathbb{P}^1_k \times \mathbb{P}^1_k$ at $w$.  
Let $t_1, t_2$ be local parameters
at $w$.  Thus $A$ is the localization of $k[t_1,t_2]$ at the maximal
ideal generated by $t_1$ and $t_2$. 
We know that $D \subset (\infty \times \mathbb{P}^1_k) \cup (\mathbb{P}^1_k \times \infty)$ consists of a bounded number of closed points.  Suppose $w \in \infty \times \mathbb{P}^1_k$.  Above, we constructed an extension, $\Psi: (\mathbb{P}^1_k \times \mathbb{P}^1_k - D) \rightarrow \mathbb{P}^h_k$, of the original morphism $\Psi$ on $\mathbb{A}^1_k \times \mathbb{A}^1_k$ that underlies $\psi$.  In particular, this construction shows that the  morphism from $\mathrm{Spec}(A) - D$ to $\mathbb{P}^h_k$ induced by $\Psi$ has the form 
$$(t_1,t_2) \to (g_0(t_1,t_2):\cdots:g_h(t_1,t_2)) = (g_0:\cdots:g_h),$$
where $g_i(t_1,t_2) \in k(t_1,t_2)$ are quotients of polynomials of bounded degrees in $t_1,t_2$.  In fact, we can bound the denominators of the $g_i(t_1,t_2)$ by writing
the entries of the matrix $Y$ as ratios of polynomials in $t_1$ and $t_2$ and by bounding
the denominators of the entries of $Y$. Since the number of entries of $Y$ is $(d')^2$,
we conclude that there is a non-zero polynomial $p = p(t_1,t_2) \in k[t_1,t_2]$ of bounded degree
such that $p(t_1,t_2) g_i(t_1,t_2) \in k[t_1,t_2]$ for all $i$.

The pullback of $\mathcal{O}_{\mathbb{P}^h_k}(1)$ to $\mathrm{Spec}(A) - w$
is a line bundle $\mathcal{L}$ on $\mathrm{Spec}(A) - w$.   Since $w$
has codimension $2$ in the regular scheme $\mathrm{Spec}(A)$,
the Weil divisor class groups of $\mathrm{Spec}(A) - w$ and $\mathrm{Spec}(A)$
are the same, and these are trivial because $A$ is a regular local ring.
Hence $\mathcal{L}$ is trivial, and we can identify $\mathcal{L}$ with
the structure sheaf $\mathcal{O}$ of $\mathrm{Spec}(A) - w$.  The  pullbacks of the coordinate
global sections of $\mathcal{O}_{\mathbb{P}^h_k}(1)$ define elements $q_0,\ldots,q_h$ of
$\Gamma(\mathrm{Spec}(A) - w, \mathcal{O}) = \Gamma(\mathrm{Spec}(A),\mathcal{O}) = A$
which generate $\mathcal{O}$ at every point of $\mathrm{Spec}(A) - w$.
Thus $q_0,\ldots,q_h$ are elements of $A$ and the $A$-ideal $I$ generated by $q_0,\ldots,q_h$
 has localization $A_P$ at every prime ideal $P$ of $A$ different from the maximal ideal corresponding
 to $w$.    
 
 We conclude that the maps from $\mathrm{Spec}(A) - D$ to $\mathbb{P}^h_k$ which
 are defined by $$(t_1,t_2) \to (q_0:\cdots:q_h)$$ and $$(t_1,t_2) \to (g_0:\cdots:g_h)$$ agree.  
 We know that the $g_i$ are ratios of polynomials in $k[t_1,t_2]$ of bounded degree, $pg_i \in k[t_1,t_2]$
 for a non-zero $p = p(t_1,t_2) \in k[t_1,t_2]$ of bounded degree, and that 
  the $q_i$ are elements of $A$ with the property that no irreducible element of 
 $A$ divides all of the $q_i$.  Here $A$ is a UFD with irreducibles equal to the irreducibles
 in $k[t_1,t_2]$ which have zero constant term, and every irreducible in $k[t_1,t_2]$
 with non-zero constant term is a unit in $A$.  
 
 By multiplying all of the $q_i$ by a suitable
 unit in $A$, we can assume that all the $q_i$ are elements of $k[t_1,t_2]$. We can furthermore
 write 
 $$(g_0,\ldots,g_h) = p^{-1}\cdot ( \ell_0,\ldots,\ell_h)$$
 where $p, \ell_0,\ldots \ell_h$ in $k[t_1,t_2]$ have bounded degrees. 
 We conclude that there must be a non-zero element
 $H$ of $\mathrm{Frac}({A}) = k(t_1,t_2)$ such that 
 \begin{equation}
 \label{eq:equals}
 H\cdot (q_0,\ldots,q_h) = (g_0,\ldots,g_h) = p^{-1}\cdot (\ell_0,\ldots,\ell_h)
 \end{equation}
 as tuples of elements of $\mathrm{Frac}({A})$.  
 
 Let $\Omega$ be the finite set of  irreducible elements $\pi$ of $k[t_1,t_2]$
 which divide either $p$ or one of $\ell_0,\ldots,\ell_h$. If $\pi$ is an irreducible
 of $k[t_1,t_2]$ which is not in $\Omega$, then $p^{-1} \ell_j = H q_j$ has valuation zero at
 the discrete valuation $\mathrm{ord}_\pi$ of $k(t_1,t_2)$ associated to $\pi \in k[t_1,t_2]$.  
 Hence 
 \begin{equation}
 \label{eq:edef}
-\mathrm{ord}_\pi(H) = \mathrm{ord}_\pi(q_0) = \cdots = \mathrm{ord}_\pi(q_h) \quad \mathrm{if}\quad \pi \not \in \Omega.
 \end{equation}
 If $\pi$ has non-zero constant term, then it is a unit in $A$, and we can
  multiply each of the $q_i$ by $\pi^{\mathrm{ord}_\pi(H) }$
  to be able to assume without loss of generality that $\mathrm{ord}_\pi(q_i)  = 0$.  If $\pi$ has zero constant term, then 
  $\pi$  defines an
  irreducible of $A$, and we know that no irreducible element of $A$ divides every one of 
  $q_0,\ldots,q_h \in A$.  So for $\pi$ with zero constant term such that $\pi \not \in \Omega$
  we conclude that $\mathrm{ord}_\pi(q_j) = 0$ for some $j$, so (\ref{eq:edef}) shows
  $\mathrm{ord}_\pi(q_i)  = 0$ for all $i$. We conclude that without loss of generality, we can assume that the
  only irreducible elements $\pi$ of $k[t_1,t_2]$ which occur in the factorizations of 
  $H$ or in one of the $q_j$ are $\pi$ lying in $\Omega$.    
  
  Suppose now that $\pi \in \Omega$. Then $\mathrm{ord}_\pi(p)$ and $\mathrm{ord}_\pi(\ell_j)$ are bounded for all $j$.
We conclude from (\ref{eq:equals}) that $\mathrm{ord}_\pi(H q_j)$ is bounded 
independently of $j$, so it follows that there is a bound on $|\mathrm{ord}_\pi(q_j) - \mathrm{ord}_\pi(q_i)|$
for all $i$ and $j$.  Suppose now that $\pi \in \Omega$ has no constant term, so that $\pi$ defines
an irreducible in $A$.  Then as above, we conclude that $\mathrm{ord}_\pi(q_j) \ge 0$ with
equality for at least one $j$.  It follows that we have an effective upper bound on $\mathrm{ord}_\pi(q_i)$
for all $i$ in this case.  Suppose now that $\pi \in \Omega$ has non-zero constant term, so that it is a unit
in $A$.  Then since $|\mathrm{ord}_\pi(q_j) - \mathrm{ord}_\pi(q_i)|$ is bounded, we can
multiply all of the $q_j$ by a suitable power of $\pi  \in A^*$ to be able
to assume that $\mathrm{ord}_\pi(q_j) \ge 0$ for all $j$ and that we have an upper bound on $\mathrm{ord}_\pi(q_j)$
for all $j$.  

  We conclude that without loss of generality, we may assume that all of the $q_j$ are in $k[t_1,t_2]$,
  that the only irreducibles in $k[t_1,t_2]$ up to associates which divide any of the $q_j$ lie in the finite set $\Omega$,
  and that the powers to which these irreducibles divide the $q_j$ are bounded.  Since
  the degrees of the elements of $\Omega$ are bounded, this implies that all the $q_j$ are now
  polynomials of bounded degree.  We know that there is no irreducible in $k[t_1,t_2]$ which
  has zero constant term and which divides all of the $q_j$.   Thus no irreducible in $A$
  divides all of the $q_j$. 
  
  Let $I$ be the ideal of $A$ generated by $q_0,\ldots,q_h$.    We will
  show
  that $A/I$ has finite bounded length, where the only simple module for the local ring $A$
  is its residue field $k$.   Let $b(q_i)$ be the sum of the exponents of the irreducibles in $A$ appearing
  in the factorization of $q_i$, where $b(q_i) = 0$ if $q_i$ is a unit.  
  We will use induction on $\beta_0= \mathrm{min}(b(q_0),\ldots,b(q_h))$,
  where $\beta_0$ is bounded.  If $\beta_0 = 0$ then some $q_i$ is a unit, and $A/I = \{0\}$.
  Suppose now that $\beta_0  > 0$.  After renumbering the $q_i$, we can suppose that $\beta_0 = b(q_0)$.  Let  $\pi \in k[t_1,t_2]$ be an irreducible in $A$ which divides $q_0$.  
  Let $I'$ be the ideal generated by $q_0/\pi, q_1,\ldots,q_h$.  
  We have an exact sequence
  \begin{equation}
  \label{eq:exact1}
  0 \to I'/I \to A/I \to A/I' \to 0
  \end{equation}
  Since $\pi$ is an irreducible in $A$ which divides $q_0$, there must be some $j >0$
  such that $\pi$ does not divide $q_j$.  
  Thus we get a surjection
  $$A/(A\pi + A q_j) \xrightarrow{\delta} I'/I \xrightarrow{} 0$$
  in which $\delta([\alpha]) = \alpha q_0/\pi$ mod $I$.
  Since $\pi \in k[t_1,t_2]$ divides $q_0$,  
  $\pi$ has bounded degree, as does $q_j$. The length
  of $A/(A\pi + A q_j)$ is the local intersection number of the divisors associated to $\pi$ and
  to $q_j$, and this length is a bounded function of the degrees of $\pi$ and $q_j$.  Since
  the length of $A/I'$ is finite and bounded by induction, we conclude from   (\ref{eq:exact1})
  that the length of $A/I$ is finite and bounded.
 
 We can now filter $A/I$ by the images of powers of the maximal ideal $m_A$ of $A$.
 By Nakayama's Lemma, this filtration is strictly decreasing until it reaches $\{0\}$.
 Thus $I$ contains a bounded power of $m_A$, and we are done with Step 1.
 \hfill $\Box$
  
  \vskip .2in

\noindent {\bf Step 2:} 
\textit{As in Step 1, let $w$ be an element of the finite set $D$ of closed points of 
$\mathbb{P}^1_k \times \mathbb{P}^1_k$
at which the rational map 
$\xymatrix{\psi:\mathbb{P}^1_k \times \mathbb{P}^1_k \ar@{-->}[r]& \mathbb{P}^h_k}$ is not defined.
Define $A_w = \mathcal{O}_{\mathbb{P}^1_k \times \mathbb{P}^1_k,w}$. 
By Step 1, the restriction of $\psi$ to $\mathrm{Spec}(A_w) - \{w\}$
is defined by polynomials $q_0 = q_0(t_1,t_2),\ldots,q_h = q_h(t_1,t_2)$ of bounded degree 
in the polynomial ring $k[t_1,t_2]$ associated to a pair of local parameters $t_1, t_2$ at $w$,
where $A_w = k[t_1,t_2]_{(t_1,t_2)}$.  
Let $I_w$ be the ideal of $A_w $ generated by the $q_j$, so that $I_w$ contains a bounded
power of the maximal ideal of $A_w$ by Step 1.  Let $\mathcal{F}$ be the coherent sheaf of ideals 
on $\mathbb{P}^1_k \times \mathbb{P}^1_k$ which has stalk $\mathcal{O}_{\mathbb{P}^1_k \times \mathbb{P}^1_k,z}$
at each point $z$ not in $D$, and whose stalk at $w \in D$ is $I_w$.  Define $B$ to be the blow-up
of $\mathbb{P}^1_k \times \mathbb{P}^1_k$ at $\mathcal{F}$. Then there is a canonical projective
birational morphism $\theta:B \to \mathbb{P}^1_k \times \mathbb{P}^1_k$ which induces an isomorphism
on the complement of $D$. For $w \in D$, the inverse image $\theta^{-1}(w)$ is a connected one dimensional scheme.  Finally,
there is a morphism $\tilde {\psi}:B \to \mathbb{P}^h_k$ which resolves the birational
map 
$\xymatrix{\psi:\mathbb{P}^1_k \times \mathbb{P}^1_k \ar@{-->}[r]& \mathbb{P}^h_k}$ in the sense that
$\tilde \psi(b) = \psi(\theta(b))$ if $\theta(b)  \not \in D$ and  $\tilde {\psi}(B)$ is a closed subset
of $\mathbb{P}^h_k$ which coincides with the orbit closure $\overline{\mathrm{Aut}_\Lambda(P) \cdot C}$
when we give each of these sets the reduced induced structure.}
\begin{figure}[ht] \caption{\label{fig:Step2} Illustration of Step 2.}
$$\xymatrix{B\ar[d]_(.45){\theta}\ar[drrr]^{\tilde{\psi}}\\ 
\mathbb{P}^1_k \times \mathbb{P}^1_k\ar@{-->}[rrr]^(.45)\psi&&&
\overline{\mathrm{Aut}_\Lambda(P) \cdot C}\;\ar@{^(->}[r]&
\mathbb{P}^h_k
}$$
\end{figure}
\medbreak

\noindent {\bf Proof of Step 2:} 
Let $U = (\mathbb{P}^1_k \times \mathbb{P}^1_k) - D$, so that $\psi$ is defined on $U$
and $\mathbb{A}^1_k \times \mathbb{A}^1_k$ is an open dense subset of $U$.  If $z \in \psi(U)$
and $V$ is an open subset of $\mathbb{P}^h_k$ containing $z$, then $\psi^{-1}(V)$ is an open 
neighborhood of $z$ in $U$. Then $\psi^{-1}(V) \cap (\mathbb{A}^1_k \times \mathbb{A}^1_k) \ne \emptyset$
and $V$ contains a point of the orbit $\psi(\mathbb{A}^1_k \times \mathbb{A}^1_k) = \mathrm{Aut}_\Lambda(P). C$.
Thus $\psi(U)$ is contained in the orbit closure $\overline{\mathrm{Aut}_\Lambda(P). C}$,
so the closure $\overline{\psi(U)}$ of $\psi(U)$ equals  $\overline{\mathrm{Aut}_\Lambda(P). C}$.

Let $G \subset U \times \mathbb{P}^h_k$ be the graph of $\psi$ on $U$.  Here $U \times \mathbb{P}^h_k$ is
an open subset of $(\mathbb{P}^1_k \times \mathbb{P}^1_k) \times \mathbb{P}^h_k$.  Then $B
= \mathrm{Proj}(\oplus_{n = 0}^\infty \mathcal{F}^n)$ is the closure 
of $G$ in $(\mathbb{P}^1_k \times \mathbb{P}^1_k) \times \mathbb{P}^h_k$. 

Since $\oplus_{n = 0}^\infty I_w^{\;n}$ is an integral domain, $B$ is an integral surface with a proper birational
 morphism $\theta:B \to \mathbb{P}^1_k \times \mathbb{P}^1_k$ that is an isomorphism away from $D$.   Namely, $\theta$ is the restriction to $B$ of the first projection morphism
$\pi_1:(\mathbb{P}^1_k \times \mathbb{P}^1_k) \times \mathbb{P}^h_k \to \mathbb{P}^1_k \times \mathbb{P}^1_k$. 
 The 
 stalk $\mathcal{F}_w$ of $\mathcal{F}$ at each $w \in D$ is contained in the maximal ideal $m_{A_w}$ of $A_w = \mathcal{O}_{\mathbb{P}^1_k \times \mathbb{P}^1_k,w}$, and $\mathcal{F}_w$ contains a positive power of $m_{A_w}$.  We claim this implies
 $\mathcal{F}_w/m_{A_w} \mathcal{F}_w$ has dimension at least $2$.  If not, then $\mathcal{F}_w/m_{A_w} \mathcal{F}_w$ has dimension $1$ by Nakayama's Lemma, from which it follows that  $\mathcal{F}_w$ is a principal $A_w$-ideal not equal to $A_w$. This is impossible
 because $A_w/\mathcal{F}_w$ has finite dimension over $k$.  Now $\theta^{-1}(w) = 
 \mathrm{Proj}(\oplus_{n = 0}^\infty (I_w^{\;n} /m_{A_w} I_w^{\;n}))$ has at least two
 points. Since $\mathbb{P}^1_k \times \mathbb{P}^1_k$ is normal, $\theta^{-1}(w)$ is connected by Zariski's Main Theorem. It follows that $\theta^{-1}(w)$ is a connected union of  finitely many possibly non-reduced curves.

We define
$\tilde{\psi}:B \to \mathbb{P}^h_k$ to be the restriction to $B$ of the second  projection morphism
$\pi_2:(\mathbb{P}^1_k \times \mathbb{P}^1_k) \times \mathbb{P}^h_k \to \mathbb{P}^h_k$.  Clearly
$\tilde{\psi}(B) = \pi_2(B) = \pi_2(\overline{G})$ contains $\pi_2(G) = \psi(U)$.  Because $\pi_2$
is projective, $\pi_2(B)$ is closed, so 
\begin{equation}
\label{eq:firstc}
\overline{\psi(U)} = \overline{\mathrm{Aut}_\Lambda(P). C} \subset \pi_2(B) = \tilde{\psi}(B).
\end{equation}

Suppose $z \in \mathbb{P}^h_k$ does not lie in
 $\overline{\mathrm{Aut}_\Lambda(P) . C} = \overline{\psi(U)}$.  Then there is an open
neighborhood $V$ of $z$ in $\mathbb{P}^h_k$ so $V \cap \psi(U) = \emptyset$. 
Therefore $\pi_2^{-1}(V)$ is an open subset of $(\mathbb{P}^1_k \times \mathbb{P}^1_k) \times \mathbb{P}^h_k$ which
contains no element of the graph $G$ of $\psi$ on $U$.  Hence $\pi_2^{-1}(z)$ is not in the closure $B$ of
$G$, so $z \not \in \pi_2(B) = \tilde{\psi}(B)$. Taking contrapositives, we have shown 
\begin{equation}
\label{eq:secondc}
 \tilde{\psi}(B) \subset \overline{\psi(U)} = \overline{\mathrm{Aut}_\Lambda(P). C}.
\end{equation} 
Combining (\ref{eq:firstc}) and (\ref{eq:secondc}) shows 
$\tilde{\psi}(B) = \overline{\mathrm{Aut}_\Lambda(P) . C}$ so we are done with Step 2.
\hfill $\Box$

 \vskip .2in

\noindent {\bf Notation for Step 3:} 
Fix $w \in D$ and let $t_1$ and $t_2$ be uniformizing parameters in $A=A_w = \mathcal{O}_{\mathbb{P}^1_k \times \mathbb{P}^1_k,w}$.  Thus $A = k[t_1,t_2]_{(t_1,t_2)}$ and   
$I = \mathcal{F}_w = 
Aq_0 + \cdots +Aq_h$ for some polynomials $q_0,\ldots,q_h \in k[t_1,t_2]$ of bounded degree with constant term $0$.  We know from Step 1 that $I$ contains a bounded power of the maximal ideal $m_A$ of $A$, so
$\mathrm{dim}_k(A/I) = a$ is bounded. Thus $I$ defines a closed point on the reduction  $\mathcal{G} = \mathcal{G}_a$
of the Grassmannian 
of all ideals of $A$ of codimension $a$.  

Let $\mathfrak{X} = \mathbb{P}^2_\mathcal{G}$ be the projective plane over $\mathcal{G}$
with structure sheaf $\mathcal{O}_\mathfrak{X}$.  For all $I \in \mathcal{G}$, let $k(I)$ be the residue field 
of $I$. By the construction of the Grassmannian, $I$ is associated to an ideal, which we will also
denote by $I$, in $k(I)[t_1,t_2]_{(t_1,t_2)}$. Note that $I$ is a closed point if and only if $k(I) = k$.
The fiber $\mathfrak{X}_I$ is $\mathbb{P}^2_{k(I)}$, which contains the 
affine plane $\mathbb{A}^2_{k(I)} = \mathrm{Spec}(k(I)[t_1,t_2])$
with origin $z_0$ defined by $t_1 = t_2 = 0$.  Let $\mathcal{I}_I \subset \mathcal{O}_{\mathfrak{X}_I}$ be 
the sheaf of ideals
which has stalk  $\mathcal{O}_{\mathfrak{X}_I,z}$ at $z \in \mathbb{P}^2_{k(I)} - \{z_0\}$ and stalk at 
$z_0 = (0,0)$ given by
the ideal  $I$ in $k(I)[t_1,t_2]_{(t_1,t_2)}$.  
Define  $\mathcal{J} \subset \mathcal{O}_\mathfrak{X}$ to be the sheaf of ideals which is generated by 
the inverse images of $\mathcal{I}_I\subset \mathcal{O}_{\mathfrak{X}_I}$ under the reduction maps 
$\mathcal{O}_\mathfrak{X} \to k(I) \otimes_{\mathcal{O}_\mathcal{G}} \mathcal{O}_\mathfrak{X} = 
\mathcal{O}_{\mathfrak{X}_I}$
associated to each  $I \in \mathcal{G}$.  Define $\mathfrak{B}$ to be the blow-up of
$\mathfrak{X}$ along the sheaf of ideals $\mathcal{J}$.  

The natural morphism $\xi:\mathfrak{B} \to \mathfrak{X} = \mathbb{P}^2_\mathcal{G}$
is projective and an isomorphism away from the zero section 
$s_0:\mathcal{G} \to  \mathfrak{X} = \mathbb{P}^2_{{\mathcal{G}}}$, which is
the section of $\mathfrak{X} = \mathbb{P}^2_{{\mathcal{G}}} \to {\mathcal{G}}$ defined by the origin on $\mathbb{A}^2_{{\mathcal{G}}} \subset \mathfrak{X}$.  
The fiber of $\xi$ over
$I \in \mathcal{G}$ is the blow-up $B_I \to \mathbb{P}^2_{k(I)}$ of $\mathbb{P}^2_{k(I)}$ at the ideal which
is the structure sheaf of $\mathbb{P}^2_{k(I)}$ outside of the origin on $\mathbb{A}^2_{k(I)}$ and the ideal  $I$ at the origin.  By a monoidal transformation
of a regular surface over a (possibly not algebraically closed) field, we will mean a blow-up of the surface at the maximal ideal defined by a closed point.
\vskip .1in

\noindent \textbf{Step 3:} \textit{Let $\mathcal{G}$ be as above. To prove Theorem \ref{thm:main}, it will suffice to show that there is a uniform bound (which may depend on $\mathcal{G}$)  on the
number of successive monoidal transformations one must perform, in order to birationally transform $\mathbb{P}^2_{k(I)}$ to a smooth rational surface which has a projective birational morphism to $B_I$,  as $I$ varies over all 
points of $\mathcal{G}$. }  
\medbreak

\noindent {\bf Proof of Step 3:}
By Step 1, the set $D$ is a finite set of closed points, and the number of points in $D$ is bounded 
(always in terms of $d' = \mathrm{dim}_k(C)$). Suppose there is a uniform bound of the kind described in 
Step 3 for each point $w \in D$ when we take $I = \mathcal{F}_w$ as in the notation just prior to Step 3, 
so that $I$ is an ideal of the local ring $A=A_w$ of $w$ on $\mathbb{P}^1_k \times \mathbb{P}^1_k$ 
which contains a bounded power of the maximal ideal. We patch together the resulting
sequences of monoidal transformations as $w$ varies over $D$.  Because the number of points in $D$ is 
bounded,
and the number of monoidal transformations needed for each point $w$ of $D$ is also bounded by the 
hypothesis in Step 3, we have the following
conclusion.  There is a bound
 on the number of successive monoidal transformations
one must make, beginning with $\mathbb{P}^1_k \times \mathbb{P}^1_k$, to arrive at a smooth rational surface $S$
which has a projective birational morphism to the blow-up $B$ described in Step 2. Let
$\mu:S \to \mathbb{P}^1_k \times \mathbb{P}^1_k$ be this morphism.

There is a projective birational morphism from $B$ to the orbit closure $\overline{\mathrm{Aut}_\Lambda(P) . C}$.
So we arrive at such a morphism from $S$ to $\overline{\mathrm{Aut}_\Lambda(P) . C}$. Thus
there is a projective birational morphism from $S$ to the minimal desingularization 
$\overline{\mathrm{Aut}_\Lambda(P) . C}^\dagger$ of the
normalization $\overline{\mathrm{Aut}_\Lambda(P) . C}^\#$ of $\overline{\mathrm{Aut}_\Lambda(P) . C}$.
Suppose $\overline{\mathrm{Aut}_\Lambda(P) . C}^\dagger \to X$ is a projective birational
morphism to a relatively minimal smooth rational projective surface $X$, so that $X$ is isomorphic to $\mathbb{P}^2_k$,
$\mathbb{P}^1_k \times \mathbb{P}^1_k$ or to one of the surfaces $X_n$ with $n \ge 2$.
We obtain a projective birational morphism $\tilde{\mu}:S \to X$ which is a composition of an unknown number of blow-downs
of rational curves of self-intersection $-1$.  
We obtain the diagram in Figure \ref{fig:Step3} (see also Figure \ref{fig:Step2}).
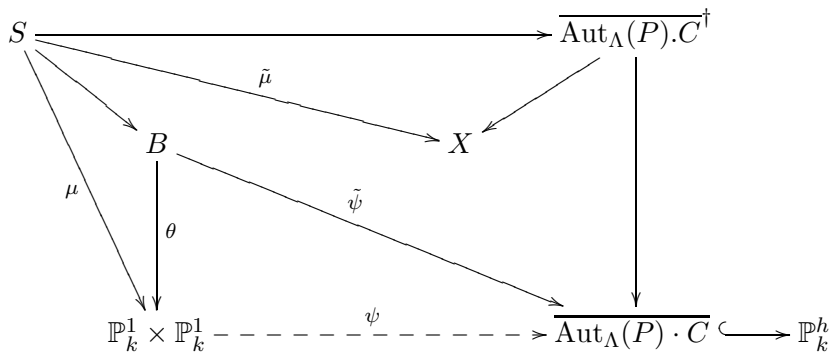
\begin{figure}[ht] \caption{\label{fig:Step3} Illustration of the proof of Step 3.}
$$\xymatrix{
S\ar[dr]\ar[rrrrr]\ar[dddr]_{\mu}\ar[drrrr]^(.55){\tilde{\mu}}&&&&&\overline{\mathrm{Aut}_\Lambda(P) . C}^\dagger\ar[ld]\ar[ddd]\\
&B\ar[dd]^(.45){\theta}\ar[ddrrrr]^(.4){\tilde{\psi}}&&&X\\ \\
&\mathbb{P}^1_k \times \mathbb{P}^1_k\ar@{-->}[rrrr]^(.45)\psi&&&&
\overline{\mathrm{Aut}_\Lambda(P) \cdot C}\;\ar@{^(->}[r]&
\mathbb{P}^h_k
}$$
\end{figure}

By \cite[Prop. V.3.2]{Hartshorne}, the rank of the Picard group of a smooth projective
surface increases by $1$ when one performs a monoidal transformation.  Hence the rank of $\mathrm{Pic}(S)$
is bounded because $\mu:S \to \mathbb{P}^1_k \times \mathbb{P}^1_k$ involves a bounded number
of blow-downs. This implies there is a bound on the number of blow-downs needed to factor
$\tilde{\mu}: S \to X$.  

We define an irreducible effective smooth curve $L$ on $X$
in the following way.  If $X = \mathbb{P}^2_k$ let $L$ be any projective line on $X$,
so that $L$ has self-intersection $L \cdot L = 1$.  If $X = \mathbb{P}^1_k \times \mathbb{P}^1_k$,
let $L = \mathbb{P}^1_k \times \{x\}$ for a point $x \in \mathbb{P}^1_k$, so that $L \cdot L = 0$.  Finally,
if $X = X_n$ for some $n \ge 2$, there is by \cite[Example V.2.11.3]{Hartshorne} a smooth
curve $L$ on $X_n$ such that $L \cdot L = -n$. 

By \cite[Prop. 3.2]{Hartshorne},
if $S'' \to S'$ is a monoidal transformation of smooth projective surfaces, and $L'$ is a smooth curve on 
$S'$ with proper transform $L''$ on $S''$, then either $L'' \cdot L'' = L' \cdot L'$
or $L'' \cdot L'' = L' \cdot L' + 1$.  Let $L_S$ be the proper transform of $L$  under $\tilde{\mu}:S \to X$,
and suppose $\tilde{\mu}$ can be factored into $\tilde{c}$ blow-downs, where we have bounded $\tilde{c}$.
Then  $(L_S) \cdot (L_S) \le L \cdot L + \tilde{c}$. 

We claim that it is impossible that $X = X_n$ for an $n > \tilde{c}+1$.  For otherwise,
$(L_S) \cdot (L_S) \le L \cdot L + \tilde{c} = -n + \tilde{c} < -1.$  If a curve is sent to a 
point by a blow-down, then its self-intersection must be $-1$, and if it is not
sent to a point then its self-intersection cannot increase. Hence we
would have that the image $L_0$ of $L_S$ under $\mu:S \to \mathbb{P}^1_k \times \mathbb{P}^1_k$
would be an effective curve of self-intersection $L_0\cdot L_0 < -1$. There are no such
curves, since the self-intersection pairing is non-negative on the cone of effective
curves in $\mathbb{P}^1_k \times \mathbb{P}^1_k$. This shows that the $n$ for
which $X$ could be isomorphic to $X_n$ are bounded. Since we have also bounded
the number of blow-downs needed to factor the morphism $\tilde{\mu}:S \to X$, this completes
Step 3.  \hfill $\Box$

\vskip .2in

\noindent{\bf Step 4:} 
\textit{Let $\xi:\mathfrak{B} \to \mathfrak{X} = \mathbb{P}^2_{\mathcal G}$ be as in Step 3. Let $\mathcal{C}$
be an irreducible component of $\mathcal{G}$, and let $\mathfrak{B}_{\mathcal{C}}$ be the pullpack of $\mathcal{G}$
to ${\mathcal{C}}$.  Define $s: {\mathcal{C}} \to \mathfrak{X}_{\mathcal{C}} = \mathbb{P}^2_{{\mathcal{C}}}$
to be the section of $\mathfrak{X}_{\mathcal{C}} = \mathbb{P}^2_{{\mathcal{C}}} \to {\mathcal{C}}$ defined by the origin on $\mathbb{A}^2_{{\mathcal{C}}} \subset \mathfrak{X}_{\mathcal{C}}$.  
Define $\mathfrak{B}_{\mathcal{C}}^\sharp$ to be the normalization of $\mathfrak{B}_{\mathcal{C}}$. Then $\mathfrak{B}_{\mathcal{C}}^\sharp \to \mathfrak{B}_{\mathcal{C}}$
is finite and an isomorphism off of  $\xi^{-1}(s(\mathcal{C}))$, where 
$s(\mathcal{C})$ is the image of the section $s: {\mathcal{C}} \to \mathfrak{X}_{\mathcal{C}}$.  The fiber $(\mathfrak{B}_{\mathcal{C}}^\sharp)_{\eta_{\mathcal{C}}}$ of $\mathfrak{B}_{\mathcal{C}}^\sharp$ over the generic point $\eta_{\mathcal{C}}$ of ${\mathcal{C}}$  is a normal projective surface over the residue field $k(\eta_{\mathcal{C}})$.}

\medbreak

\noindent {\bf Proof of Step 4:}  The local rings of $\mathfrak{B}_{\mathcal{C}}^\sharp$ are integrally closed in 
their fraction fields. For points $z$ on  $(\mathfrak{B}_{\mathcal{C}}^\sharp)_{\eta_{\mathcal{C}}}$, the local
ring of $z$ on $\mathfrak{B}_{\mathcal{C}}^\sharp$ is equal to the local ring of $z$ on $(\mathfrak{B}_{\mathcal{C}}^\sharp)_{\eta_{\mathcal{C}}}$.  Thus
$(\mathfrak{B}_{\mathcal{C}}^\sharp)_{\eta_{\mathcal{C}}}$ is a normal projective surface, and the rest of Step 4 is clear.
\hfill $\Box$

\vskip .2in

\noindent {\bf Step 5:} \textit{In this step we show that there is a uniform bound as specified in Step 3, 
which will complete the proof of Theorem \ref{thm:main}.  
The plan is to apply a Theorem of Zariski to the general fiber 
$(\mathfrak{B}_{\mathcal{C}}^\sharp)_{\eta_{\mathcal{C}}}$ of $\mathfrak{B}_{\mathcal{C}}^\sharp$.
This produces a bound of the kind required in Step 3 when the point $I$ of $\mathcal{G}$ is the generic 
point
of the irreducible component ${\mathcal{C}}$ of $\mathcal{G}$.  We then show how to extend the 
monoidal transformations
involved over an open dense subset of ${\mathcal{C}}$, to be able to handle all $I$ in such a subset.  
Since there
are a bounded number of irreducible components ${\mathcal{C}}$ of $\mathcal{G}$, this provides us 
with a bound of the required
kind for  all $I$
in an open dense subset of $\mathcal{G}$.  Finally, we apply Noetherian
induction to deal with all $I$ in $\mathcal{G}$.}
\medbreak

\noindent {\bf Proof of Step 5:} To begin the proof, 
we recall a Theorem of Zariski concerning birational
projective morphisms $f:Z' \to Z$ between projective normal surfaces over a field $F$.
Such morphisms are called modifications by Artin in \cite{Artin}.  

A normalized monoidal tranformation of $Z$ is a morphism $Z_1 \to Z$ in which $Z_1$
is the normalization of the monoidal transformation of $Z$ at the maximal ideal of a closed point.
Zariski proves that every modification $f:Z' \to Z$ is dominated by a 
morphism $g:Z'' \to Z$ which is formed as the composition of a finite series
of normalized monoidal transformations.  Here domination means that there is a birational
morphism $h:Z'' \to Z'$ compatible with $f:Z' \to Z$ and $g:Z'' \to Z$.  

With the notations of Step 4 we now apply Zariski's Theorem to the morphism
$$f:Z' = (\mathfrak{B}_{\mathcal{C}}^\sharp)_{\eta_{\mathcal{C}}} \to Z = \mathbb{P}^2_{\eta_{\mathcal{C}}}$$ with
$F = k(\eta_{\mathcal{C}})$.  We obtain a birational morphism $$g:Z'' \to Z$$ of normal projective
surfaces over $k(\eta_{\mathcal{C}})$ which is the composition
of a finite number of normalized monoidal transformations together with a birational
morphism $h:Z'' \to Z'$ which is compatible with $f$ and $g$.  Since 
$Z = \mathbb{P}^2_{\eta_{\mathcal{C}}}$ is regular, and monoidal tranformations
preserve regularity, we never have to normalize any of the results of the
monoidal transformations used in constructing $g:Z'' \to Z$.  

Define 
$\mathcal{Z} = \mathbb{P}^2_{\mathcal{C}}$, so that the generic fiber of $\mathcal{Z}$ over
${\mathcal{C}}$ is $Z$.  Take a closed point $Q$ on $Z$,
and let $\overline{Q}$ be the Zariski closure of $Q$ in $\mathcal{Z}$.
The scheme $\overline{Q}$ need not be reduced, since $Q$ might
have residue field $k(Q)$ which is a finite inseparable extension 
of $k(\eta_{\mathcal{C}})$. But if we let $\overline{Q}^{red}$ be the reduction
of $\overline{Q}$, then $\overline{Q}^{red}$ is finite over ${\mathcal{C}}$
whose fibers over each point $I$ of ${\mathcal{C}}$ are a disjoint union of reduced
points which are closed in the fiber over $I$.  The number of such points in each fiber is 
bounded by the relative degree $[k(Q):k(\eta_{\mathcal{C}})]$.  

Let
$\mathcal{Z}_1$ be the blow-up of $\mathcal{Z}$ at the
sheaf of ideals which defines $\overline{Q}^{red}$.  
Since $Q$ is a closed point on the general fiber $Z$
of $\mathcal{Z}$, $\overline{Q}^{red}$ has general
fiber $Q$.  We conclude that $\mathcal{Z}_1$ has as general fiber the
blow-up $Z_1$ of $Z$ at $Q$. Furthermore, each 
fiber of $\mathcal{Z}_1$ over ${\mathcal{C}}$ is the blow-up
of the fiber of $\mathcal{Z}$ at the disjoint union of a
finite number of reduced closed points, and the number
of such points is uniformly bounded by $[k(Q):k(\eta_{\mathcal{C}})]$.
We conclude that each fiber of $\mathcal{Z}_1$ over ${\mathcal{C}}$
results from a number of monoidal transformations
of the corresponding fiber of $\mathcal{Z}$ over ${\mathcal{C}}$, with the number
of transformations involved being bounded by $[k(Q):k(\eta_{\mathcal{C}})]$.

We apply this process to each of the monoidal transformations
involved in producing the morphism $g:Z'' \to Z$.  We conclude
that $g$ is the general fiber over ${\mathcal{C}}$ of a projective
birational morphism $\tilde{g}:\mathcal{Z}'' \to \mathcal{Z} = \mathbb{P}^2_{\mathcal{C}}$
which has the following property.  For each point $I$ of ${\mathcal{C}}$, the fiber
$\tilde{g}_I:\mathcal{Z}''_I \to \mathcal{Z}_I = (\mathbb{P}^2_{{\mathcal{C}}})_I=
\mathbb{P}^2_{k(I)}$ of $\tilde{g}$ at $I$ is 
an isomorphism on the complement of
the origin of $\mathbb{A}^2_{k(I)} \subset \mathbb{P}^2_{k(I)}$, and $\tilde{g}_I$ is 
the composition of a bounded number of
monoidal tranformations.    

Recall now that we have a proper birational morphism $$h:Z'' \to  
Z' = (\mathfrak{B}_{\mathcal{C}}^\sharp)_{\eta_{\mathcal{C}}}$$
of normal surfaces over $k(\eta_{\mathcal{C}})$
which gives a commutative diagram 
$$\xymatrix{
Z''\ar[rd]_g\ar[rr]^h&&Z' = (\mathfrak{B}_{\mathcal{C}}^\sharp)_{\eta_{\mathcal{C}}}\ar[ld]^f\\
&Z  = \mathbb{P}^2_{\eta_{\mathcal{C}}}}$$
Here $h$ defines an isomorphism between dense open subsets of $Z''$ and $Z'$.

Let $\tau:\mathfrak{B}_{\mathcal{C}}^\sharp \to \mathfrak{B}_{\mathcal{C}}$ be the (finite) normalization morphism. 
We will use a subscript $M$ to denote the pullback  of a scheme or a morphism between schemes
over a subscheme $M$ of ${\mathfrak{B}}$. 
Then $f_{\eta_{\mathcal{C}}}:(\mathfrak{B}_{\mathcal{C}}^\sharp)_{\eta_{\mathcal{C}}} \to \mathbb{P}^2_{\eta_{\mathcal{C}}}$ is the composition
$\xi_{\eta_{\mathcal{C}}} \circ \tau_{\eta_{\mathcal{C}}}$ where $\xi:\mathfrak{B} \to \mathbb{P}^2_{{\mathcal{G}}}$ is as in the notation stated just prior to  Step 3.

By considering the denominators of the coefficients in $k(\eta_{\mathcal{C}})$ of homogeneous
polynomials defining $h$, we see that there is a dense open affine subset $\mathrm{Spec}(R)$
of ${\mathcal{C}}$ such that $h$ extends to a morphism $$\tilde{h}_R:\mathcal{Z}''_R \to \mathfrak{B}^\sharp_R$$
where $\mathcal{Z}''_R$ is the pullback to $\mathrm{Spec}(R)$ of the
scheme $\mathcal{Z}''$ constructed above, and $\mathfrak{B}^\sharp_R$ is the 
pullback of $\mathfrak{B}^\sharp$ to $\mathrm{Spec}(R)$.  We can further shrink $\mathrm{Spec}(R)$
so that $\tilde{h}_R$
is compatible with the restriction
$$\tilde{g}_R:\mathcal{Z}''_R \to \mathcal{Z}_R = \mathbb{P}^2_R$$ of 
$\tilde{g}:\mathcal{Z}'' \to \mathcal{Z} = \mathbb{P}^2_{\mathcal{C}}$ over $\mathrm{Spec}(R)$
and with the restriction 
$$(\xi\circ \tau)_R:\mathfrak{B}_R^\sharp \to \mathbb{P}^2_{R}$$
of $\xi\circ \tau$ to $\mathrm{Spec}(R)$.

We conclude that for all points $I \in \mathrm{Spec}(R)$,
we have a birational morphism $$\tilde{h}_I: \mathcal{Z}''_I \to (\mathfrak{B}_{\mathcal{C}}^\sharp)_I$$
compatible with the bounded composition of mononoidal transformations
$$\tilde{g}_I:\mathcal{Z}''_I \to \mathbb{P}^2_{k(I)} = (\mathbb{P}^2_{{\mathcal{C}}})_I$$ and the 
birational morphism $$(\xi \circ \tau)_I:(\mathfrak{B}_{\mathcal{C}}^\sharp)_I \to \mathbb{P}^2_{k(I)} = (\mathbb{P}^2_{{\mathcal{C}}})_I$$
which is the composition of $\tau_I: (\mathfrak{B}_{\mathcal{C}}^\sharp)_I \to \mathfrak{B}_I $ with the blow-up
morphism $\xi_I:\mathfrak{B}_I \to  \mathbb{P}^2_{k(I)} = (\mathbb{P}^2_{{\mathcal{C}}})_I$.  Since
$\tau_I$ is an isomorphism on the complement of the one-dimensional fiber of $\mathfrak{B}_I$
over the origin in $\mathbb{A}^2_{k(I)} \subset \mathbb{P}^2_{k(I)}$, we know that
$\tau_I$ is birational. Thus the composition of $\tilde{h}_I$ with
$\tau_I$ is a birational morphsm $$\zeta_I:\mathcal{Z}_I'' \to \mathfrak{B}_I$$
compatible with $\tilde{g}_I$ and $\xi_I$.  Since $\tilde{h}_I$
is the composition of a bounded number of monoidal transformations, we
have now produced the bound required in Step 3 for all $I$
in the open dense subset $\mathrm{Spec}(R)$ of ${\mathcal{C}}$.

We now do the above construction for the bounded number of irreducible components ${\mathcal{C}}$ of $\mathcal{G}$.
This produces a bound of the kind needed in Step 3 for all $I$ in
an open dense subset of $\mathcal{G}$. The complement of these
$I$ is a closed subset of dimension strictly smaller than the dimension of
$\mathcal{G}$. Continuing by Noetherian induction on these subsets,
we use the fact that $\mathrm{dim}(\mathcal{G})$ is bounded
to conclude that we have a bound of the required kind which applies
to all points $I$ of $\mathcal{G}$.  In view of Step 3, this completes
the proof of Theorem \ref{thm:main}.
\hfill $\Box$

\begin{Remark}
\label{rem:moregeneral}
The proof of Theorem \ref{thm:main} can be easily modified to show the following generalization:

Suppose $T$ is simple as before, but the orbit $\mathrm{Aut}_\Lambda(P) . C$ of $C$ in $\grass$ is 
$\mathfrak{m}$-dimensional for some $\mathfrak{m}\ge 2$. Let $Z$ be any two-dimensional affine plane contained in  
$\mathrm{Aut}_\Lambda(P) . C$. Let $\overline{Z}^\dagger$ be the minimal desingularization of the 
normalization of the closure $\overline{Z}$ of $Z$ in $\grass$.

There is a bound $n_0$ which depends only on $k$ and $\mathrm{dim}_k(C)$ such that
there is a birational morphism from $\overline{Z}^\dagger$ to
a relatively minimal smooth rational projective surface which is either $\mathbb{P}^2_k$ or $X_n$ for some 
integer $0 \le n \le n_0$ with $n \ne 1$.  There is furthermore a bound depending on $k$ and $\mathrm{dim}_k(C)$
for the number of irreducible curves which are blown down to points by this morphism.
\end{Remark}

\section{Bounding Euler characteristics of orbit closures}
\label{sec:eulerbounds}
\setcounter{equation}{0}

In this section we assume the hypotheses and notations of Theorem \ref{thm:main}.
To prove Corollary \ref{cor:Euler} we must show that the Euler characteristic
$\chi(\overline{\mathrm{Aut}_\Lambda(P) . C})$ is bounded above by a function 
which depends only on $k$ and $\mathrm{dim}_k(C)$.  Here we define the
Euler characteristic $\chi(V)$ of an arbitrary separated open subscheme $V$
of a closed subscheme $\overline{V}$ 
of a projective space over $k$ to be the \'etale Euler characteristic with compact
support
$$\chi(V) = \sum_i (-1)^i  \cdot \mathrm{dim}_{\mathbb{Q}_\ell} H^i_c(V,\mathbb{Q}_\ell)$$
for any  prime $\ell$ different from the characteristic of $k$ (see \cite[p. 93]{Milne}).
By \cite[Remark 1.30]{Milne}, if $Z$ is a closed subscheme of $V$, there is a long
exact sequence
$$\cdots \to H^i_c(V-Z,\mathbb{Q}_\ell) \to H^i_c(V,\mathbb{Q}_\ell) \to H^i_c(Z,\mathbb{Q}_\ell) \to \cdots$$
This gives an equality of Euler characteristics 
\begin{equation}
\label{eq:basic}
\chi(V) = \chi(Z) + \chi(V - Z).
\end{equation}
If $\mathcal{T}$ is a constructible sheaf of $\mathbb{Q}_\ell$-vector spaces on $V$
for some prime $\ell$ different from the characteristic of $k$, we will denote
the compactly supported Euler characteristic of $\mathcal{T}$ by
$$\chi(V,\mathcal{T}) = \sum_i (-1)^i \cdot \mathrm{dim}_{\mathbb{Q}_\ell} H^i_c(V,\mathcal{T}).$$
Thus
$\chi(V) = \chi(V,\mathbb{Q}_\ell).$

We showed in the first paragraph of the proof of Step 3 of the proof of Theorem \ref{thm:main} 
the following fact.  Once Step 3 was completed, there is a smooth rational surface
$S$ together with morphisms $\mu:S \to \mathbb{P}^1_k \times \mathbb{P}^1_k$ and
$\nu:S \to \overline{\mathrm{Aut}_\Lambda(P) . C}$ with the following properties.
The morphism $\mu$ is a composition of a bounded number of monoidal transformations,
and it is an isomorphism over $\mathbb{A}^2_k = \mathbb{A}^1_k \times \mathbb{A}^1_k$.
The morphism $\nu$ is proper and birational.  On  $\mu^{-1}(\mathbb{A}^2_k)$, $\nu$ is 
equal to $\Psi\circ \mu$, where $\Psi:\mathbb{A}^2_k \to \mathrm{Aut}_\Lambda(P) . C$
is the isomorphism in (\ref{eq:psi}).
\begin{equation}
\xymatrix{
&S\ar[dl]_{\mu}\ar[dr]^{\nu}\\
\mathbb{P}^1_k \times \mathbb{P}^1_k&&\overline{\mathrm{Aut}_\Lambda(P) . C}\\
\mathbb{A}^2_k\ar@{^(->}[u]\ar[rr]^{\Psi}&&\mathrm{Aut}_\Lambda(P) . C\ar@{^(->}[u]
}
\end{equation}

We write $\mathbb{P}^1_k \times \mathbb{P}^1_k - \mathbb{A}^2_k = \Delta$ as the union
of the two rational curves $\infty \times \mathbb{P}^1_k$ and $\mathbb{P}^1_k \times \infty$
crossing transversely at $\infty \times \infty$.  By induction on the number $n$ of monoidal transformations
involved in factoring $\mu:S \to \mathbb{P}^1_k \times \mathbb{P}^1_k$, we see that
$\mu^{-1}(\Delta)$ is reduced and a tree of $n+2$ smooth $\mathbb{P}^1_k$-curves crossing transversely, with each point lying on at most two irreducible components.  The normalization $\mu^{-1}(\Delta)^\#$
is thus the disjoint union of $n+2$ copies of $\mathbb{P}^1_k$. Furthermore, the canonical
morphism $\mu^{-1}(\Delta)^\# \to \mu^{-1}(\Delta)$ is an isomorphism off the $r$ points
of $\mu^{-1}(\Delta)$ which lie on two distinct irreducible components.  Since $\mu^{-1}(\Delta)$
is a tree of $\mathbb{P}^1_k$-curves, we see by induction on the number $n+2$ of irreducible
components of this tree that $r = n+1$.  We therefore have from (\ref{eq:basic}) the Euler characteristic identities
$$\chi(S) = \chi(\mu^{-1}(\mathbb{A}^2_k)) + \chi(\mu^{-1}(\Delta)) =  1 + \chi(\mu^{-1}(\Delta))$$
and
\begin{eqnarray*}
 2n + 4  &=& (n+2) \cdot \chi(\mathbb{P}^1_k)\\
 & = & \chi(\mu^{-1}(\Delta)^\#) \\
 &=& \chi(\mu^{-1}(\Delta)) + n+1\,.
\end{eqnarray*}
Thus
$$\chi(S) = n+4.$$

In a similar way, we can write the orbit closure  $\overline{\mathrm{Aut}_\Lambda(P) . C}$
as the disjoint union of $\Psi(\mathbb{A}^2_k) \cong \mathbb{A}^2_k$ with the closed subset 
$\nu(\mu^{-1}(\Delta))$.  Here $\nu( \mu^{-1}(\Delta))$ is the connected union of 
some number $t$ of distinct rational curves which may be singular, and $t \le n+2$.
By the same reasoning above, we find that
\begin{equation}
\label{eq:orbitchi}
\chi(\overline{\mathrm{Aut}_\Lambda(P) . C}) = \chi(\mathbb{A}^2_k) + \chi(\nu(\mu^{-1}(\Delta))) = 1 + \chi(\nu(\mu^{-1}(\Delta))).
\end{equation}

To bound $\chi(\nu(\mu^{-1}(\Delta)))$ we use the fact that the normalization $\nu(\mu^{-1}(\Delta))^\#$ is the disjoint union of $t$ copies of $\mathbb{P}^1_k$. Therefore 
\begin{equation}
\label{eq:normaleasy}
\chi(\nu(\mu^{-1}(\Delta))^\#) = 2t.
\end{equation}
Let 
$\sigma:\nu(\mu^{-1}(\Delta))^\# \to \nu(\mu^{-1}(\Delta))$ be the canonical morphism.
We thus have an exact sequence of \'etale sheaves on $\nu(\mu^{-1}(\Delta))$ given by
$$0 \to \mathbb{Q}_\ell \to \sigma_* \mathbb{Q}_\ell \to \mathcal{T} \to 0$$
in which $\mathcal{T}$ is supported on the finitely many singular points of $\nu(\mu^{-1}(\Delta))$.
This gives the following equality of compactly supported Euler characteristics on $\nu(\mu^{-1}(\Delta))$:
\begin{eqnarray}
\label{eq:easyEulerequal}
\chi(\nu(\mu^{-1}(\Delta))) &=& \chi(\nu(\mu^{-1}(\Delta)),\mathbb{Q}_\ell)\nonumber\\
& = &\chi(\nu(\mu^{-1}(\Delta)), \sigma_* \mathbb{Q}_\ell) - \chi(\nu(\mu^{-1}(\Delta)), \mathcal{T}).
\end{eqnarray}
 
 Since $\sigma$ is finite, the higher derived functors $R^q \sigma_* $ vanish on $\mathbb{Q}_\ell$ for $q > 0$.
Therefore the spectral sequence $$H^p(\nu(\mu^{-1}(\Delta)),R^q \sigma_* \mathbb{Q}_\ell) \Rightarrow H^{p+q}(
\nu(\mu^{-1}(\Delta))^\#,\mathbb{Q}_\ell)$$ degenerates to give 
\begin{equation}
\label{eq:down}
 \chi(\nu(\mu^{-1}(\Delta)), \sigma_* \mathbb{Q}_\ell) = \chi(\nu(\mu^{-1}(\Delta))^\#,\mathbb{Q}_\ell) = \chi(\nu(\mu^{-1}(\Delta))^\#) = 2t .
 \end{equation}

Since the higher cohomology groups of $\mathcal{T}$ vanish, we have
\begin{equation}
\label{eq:h0}
\chi(\nu(\mu^{-1}(\Delta)), \mathcal{T}) = \mathrm{dim}_{\mathbb{Q}_\ell} H^0_c(\nu(\mu^{-1}(\Delta)),\mathcal{T}).
\end{equation}
Putting (\ref{eq:down}) and (\ref{eq:h0}) into (\ref{eq:easyEulerequal}) gives 
\begin{equation}
\label{eq:almostdonemaybe}
\chi(\nu(\mu^{-1}(\Delta))) = 2t  - \mathrm{dim}_{\mathbb{Q}_\ell} H^0_c(\nu(\mu^{-1}(\Delta)),\mathcal{T}) \le 2t \le 2n + 4.
\end{equation}
This and  (\ref{eq:orbitchi}) complete the proof of Corollary \ref{cor:Euler}.
(With a little more work, one can in fact show $\chi(\nu(\mu^{-1}(\Delta))) \le t+1  \le n + 3$.)

\section{Examples}
\label{sec:examples}
\setcounter{equation}{0}

In this section, we provide some examples of closures of orbits of dimension 2. 

\begin{Example}
\label{ex:P11andP2}
In our first two examples, the orbit closures are isomorphic to $\mathbb{P}^1_k\times \mathbb{P}^1_k$
and $\mathbb{P}^2_k$, respectively.
\begin{enumerate}
\item[(a)]
Let $\Lambda=kQ/I$, where
\begin{eqnarray*}
Q &=& \quad \vcenter{\xymatrix{
1 \ar@(,ul)_{\omega_1}  \ar@(dl,d)_{\omega_2} \ar@<0.8ex>[rr]^{\alpha_1} \ar@<-0.8ex>[rr]_{\alpha_2}&&
 2}}\qquad\mbox{ and}\\
I &=& \quad\langle\omega_1^2,\omega_1\omega_2,\omega_2\omega_1,\omega_2^2,\alpha_1\omega_2,
\alpha_2\omega_1\rangle\;.
\end{eqnarray*}
Then $P=\Lambda e_1$ may be visualized by the diagram
$$P=\qquad\vcenter{\xymatrix{
&&1\ar@{-}[lld]_{\omega_1}\ar@{-}[ld]^{\omega_2}\ar@{-}[rd]_{\alpha_1}
\ar@{-}[rrd]^{\alpha_2}\\
1\ar@{-}[d]_{\alpha_1}&1\ar@{-}[d]^{\alpha_2}
&&2&2\\2&2}}$$
The action of $\mathrm{Aut}_\Lambda(P)$ on any submodule $C$ of $P$
is given by right multiplication by 
$e_1+t_1\omega_1+t_2\omega_2$
for $t_1,t_2\in k$.
For $C=\Lambda\alpha_1 + \Lambda\alpha_2 = k\alpha_1 + k \alpha_2$, 
the orbit $\mathrm{Aut}_\Lambda(P).C$ therefore consists of the points
\begin{equation}
\label{eq:fam1}
C(e_1+t_1\omega_1+t_2\omega_2) = 
k\left(\alpha_1+t_1\alpha_1\omega_1\right) + k\left(\alpha_2+t_2\alpha_2\omega_2\right) .
\end{equation}
Introducing projective coordinates $(u_0:u_1)\times (v_0:v_1)$ for $\mathbb{P}^1_k\times
\mathbb{P}^1_k$ and letting
$t_1=\frac{u_1}{u_0}$ and $t_2=\frac{v_1}{v_0}$, we find the family (\ref{eq:fam1}) of modules over 
$\mathbb{A}^2_k=\mathbb{A}^1_k\times\mathbb{A}^1_k$ 
to correspond to the affine patch, where $u_0\neq 0$ and $v_0\neq 0$, of the family 
of modules
\begin{eqnarray*}
\label{eq:fam11}
\tilde{C} &=& 
k[u_0,u_1,v_0,v_1]\left(u_0\alpha_1+u_1\alpha_1\omega_1\right) + \\
&&k[u_0,u_1,v_0,v_1]\left(v_0\alpha_2+v_1\alpha_2\omega_2\right) 
\end{eqnarray*}
over $\mathbb{P}^1_k\times\mathbb{P}^1_k$.
Sending $(u_0:u_1)\times (v_0:v_1)\in \mathbb{P}^1_k\times\mathbb{P}^1_k$ to 
$$ \Lambda\left(u_0\alpha_1+u_1\alpha_1\omega_1\right) + 
\Lambda\left(v_0\alpha_2+v_1\alpha_2\omega_2\right)$$
in $\mathfrak{Grass}^{S_1}_5$, we obtain a well-defined morphism of schemes 
$$\mathbb{P}^1_k\times\mathbb{P}^1_k \to \mathfrak{Grass}^{S_1}_5\;.$$
It follows that $\overline{\mathrm{Aut}_\Lambda(P).C}\cong \mathbb{P}^1_k\times\mathbb{P}^1_k$.
The boundary $\overline{\mathrm{Aut}_\Lambda(P).C} \, -\, \mathrm{Aut}_\Lambda(P).C$ consists of precisely three $\mathrm{Aut}_\Lambda(P)$-orbits (two of them $1$-dimensional), meaning that $M = P/C$ has precisely three top-stable degenerations, up to isomorphism.

\item[(b)]
Let $\Lambda=kQ/I$, where
$$Q = \quad \vcenter{\xymatrix{
1 \ar@(u,ul)_{\omega_1}  \ar@(dl,d)_{\omega_2} \ar[r]^{\alpha}&
 2}}\qquad\mbox{and}\qquad I=\quad\langle \omega_i\omega_j\;|\; i,j\in\{1,2\}\rangle\;.$$
 Let $P=\Lambda e_1$ and $C=\Lambda\alpha = k\alpha$. Sending $(z_0:z_1:z_2)\in\mathbb{P}^2_k$ to 
$ \Lambda\left(z_0\alpha+z_1\alpha\omega_1+z_2\alpha\omega_2\right)$
in $\mathrm{Grass}^{S_1}_5$, we obtain a well-defined morphism of schemes 
$$\mathbb{P}^2_k \to \mathfrak{Grass}^{S_1}_5\;$$
with the affine patch $z_0 \ne 0$ corresponding to $\mathrm{Aut}_\Lambda(P).C$. 
We conclude that $\overline{\mathrm{Aut}_\Lambda(P).C}\cong \mathbb{P}^2_k$.
The boundary $\overline{\mathrm{Aut}_\Lambda(P).C} \, -\, \mathrm{Aut}_\Lambda(P).C$
consists of infinitely many  $0$-dimensional $\mathrm{Aut}_\Lambda(P).C$-orbits corresponding to a 
$\mathbb{P}^1_k$-family of isomorphism classes of top-stable degenerations of $M = P/C$.  
Namely, this family is given by $\left(P / C_{(a:b)}\right)$, where 
$C_{(a:b)} = \Lambda(a \alpha \omega_1 + b \alpha \omega_2)$.
\end{enumerate}
\end{Example}

\begin{Example}
\label{ex:X2}
In our next example, the orbit closure $\overline{\mathrm{Aut}_\Lambda(P) . C}$
is isomorphic to the Hirzebruch surface $X_2$.
Let $\Lambda=kQ/I$, where
$$Q = \quad {\xymatrix{
3&&\ar[ll]_{\gamma} 1 
\ar@'{@+{[0,0]+(4,4)} @+{[0,0]+(0,15)}
@+{[0,0]+(-4,4)}}_(.4){\omega}
\ar@<0.8ex>[rr]^{\alpha} \ar@<-0.8ex>[rr]_{\beta} &&2}}
\qquad\mbox{and}\qquad I=
\quad\langle \omega^3,\beta\omega^2\rangle\;.$$
Moreover, we let $P=\Lambda e_1$ and 
$$C=\Lambda(\alpha+\beta) +\Lambda\alpha\omega +\Lambda\gamma\omega= 
k(\alpha+\beta)+k\alpha\omega +k\gamma\omega \in \mathfrak{Grass}^{S_1}_8 .$$ 
The module $M = P/C$ is displayed at the top of Figure \ref{fig:X2}.
Again, $\mathrm{Aut}_\Lambda(P) .C$ consists of the points 
\begin{eqnarray*}
C(e_1+t_1\omega+t_2\omega^2) &=& 
k\left(\alpha+\beta+t_1\alpha\omega+t_1\beta\omega+t_2\alpha\omega^2\right) +\nonumber\\
&& k\left(\alpha\omega+t_1\alpha\omega^2\right) + k\left(\gamma\omega + t_1\gamma\omega^2\right) ,
\end{eqnarray*}
for $(t_1,t_2)\in\mathbb{A}^2_k$.
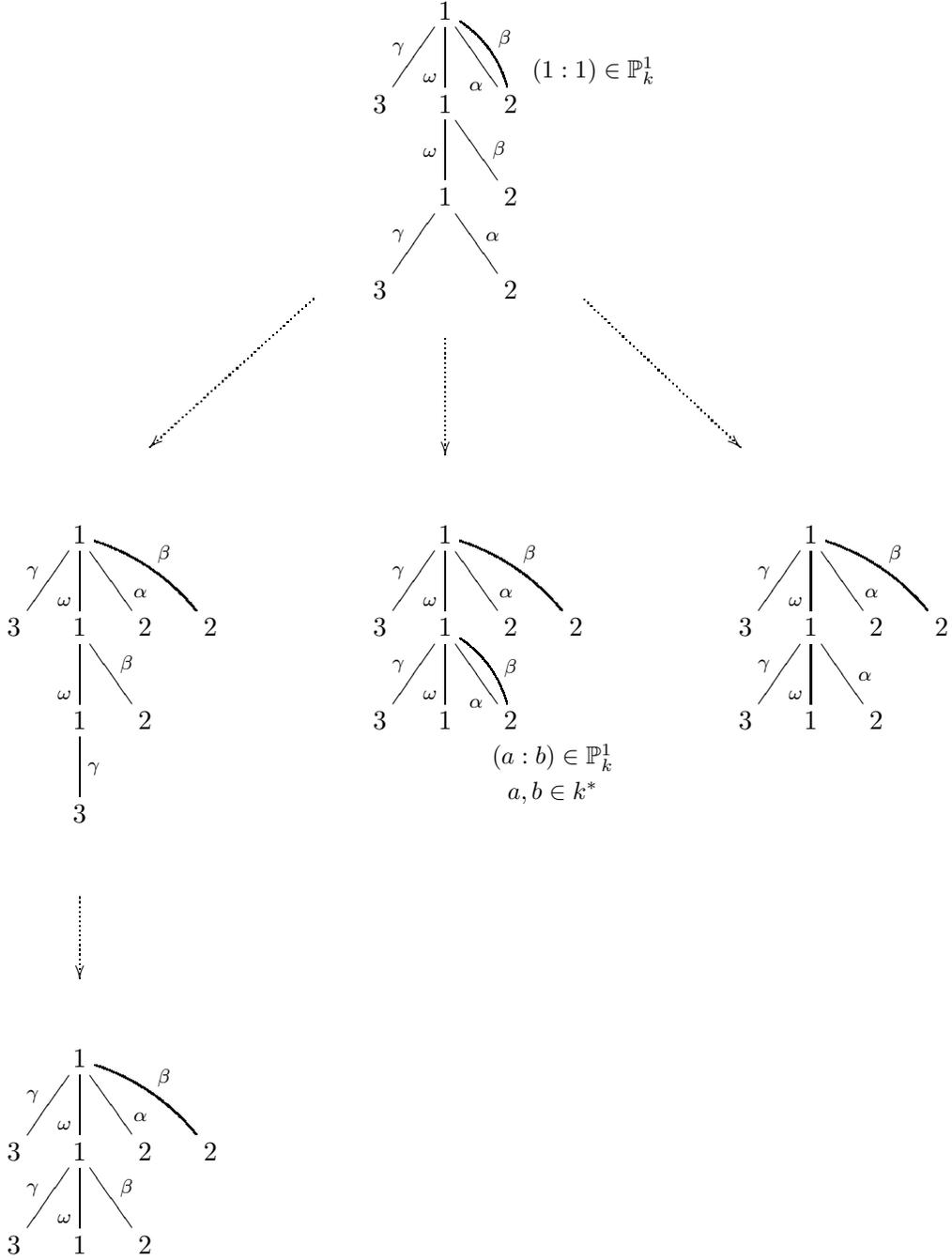
\begin{figure}[ht] \caption{\label{fig:X2} The hierarchy of top-stable degenerations  in 
Example \ref{ex:X2}.}
 $$ \xymatrixrowsep{2.0pc}\xymatrixcolsep{1.2pc}
\xymatrix{
 &&&&&&&1 \edge[dl]_{\gamma} \edge[d]_(0.7){\omega} \edge[dr]_(0.7){\alpha} \edge@/^/[dr]^{\beta}  \\
 &&&&&&3 &1 \edge[d]_{\omega} \edge[dr]^(0.6){\beta} &2 & 
 {\save+<1.5ex,3ex> \drop{\mbox{\small $(1:1) \in \mathbb{P}^1_k$}}\restore}  
\\
 &&&&&&&1 \edge[dl]_{\gamma} \edge[dr]^{\alpha} &2  \\
 &&&&& \dttdar[ddlll] &3 &   {\save[0,0]+<0ex,-4ex>  \ar@{}^{}="YY" \restore} {\save[3,0]+<0ex,7ex>  \ar@{}^{}="AA" \restore} \dttdar"YY";"AA" &2 & \dttdar[ddrrr]  \\
 &&&&&&& &&&&&&&&  \\  &&&&&&& &&&&&&&&  \\
 &1 \edge[dl]_{\gamma} \edge[d]_(0.7){\omega} \edge[dr]^(0.7){\alpha} \edge@/^/[drr]^{\beta} &&&&&&1 \edge[dl]_{\gamma} \edge[d]_(0.7){\omega} \edge[dr]^(0.7){\alpha} \edge@/^/[drr]^{\beta} &&&&&&1 \edge[dl]_{\gamma} \edge[d]_(0.7){\omega} \edge[dr]^(0.7){\alpha} \edge@/^/[drr]^{\beta}   &&  \\
3 &1 \edge[d]_(0.7){\omega} \edge[dr]^{\beta} &2 &2 &&&3 &1 \edge[dl]_{\gamma} \edge[d]_(0.7){\omega} \edge[dr]_(0.7){\alpha} \edge@/^/[dr]^(0.6){\beta}  &2 &2 &&&3 &1 \edge[dl]_{\gamma} \edge[d]_(0.7){\omega} \edge[dr]^(0.6){\alpha} &2 &2  \\
 &1 \edge[d]^{\gamma} &2 &&&&3 &1 &2 &&&&3 &1 &2  &  \\
 &3  {\save[0,0]+<0ex,-7ex> \drop{} \ar@{}^{}="BB" \restore}  {\save[3,0]+<0ex,7ex> \drop{} \ar@{}^{}="CC" \restore} \dttdar"BB";"CC" &&&&&&&& {\save+<-2ex,5ex> \drop{\mbox{\small $(a:b) \in \mathbb{P}_k^1$}} \restore} {\save+<-2ex,2ex> \drop{\mbox{\small $a,b \in k^*$}} \restore}  \\  \\  \\
 &1 \edge[dl]_{\gamma} \edge[d]_(0.7){\omega} \edge[dr]^(0.7){\alpha} \edge@/^/[drr]^{\beta}  \\
3 &1 \edge[dl]_{\gamma} \edge[d]_(0.7){\omega} \edge[dr]^{\beta} &2 &2  \\
3 &1 &2 &&&&&
}$$
\end{figure}

We will exhibit an isomorphism from 
$$X_2 = \{ (z_0: z_1) \times (y_0: y_1: y_2: y_3)  \in \mathbb{P}^1_k \times \mathbb{P}^3_k \mid z_0 y_1 = z_1 y_0 \text{\ and \ } z_0 y_2 = z_1 y_1\}$$
to $\overline{\mathrm{Aut}_\Lambda(P).C}$, which sends a point
$(z_0: z_1) \times (y_0: y_1: y_2: y_3)$ of 
$X_2$ with $z_0 \ne 0$ and $y_0 \ne 0$ to the point
$$C \left(e_1 + \frac{z_1}{z_0} \omega \ + \ \frac{y_3}{y_0} \omega^2 \right) \  \in \ \mathrm{Aut}_\Lambda(P).C.$$
In other words, we will check that the Zariski-continuous extension $\varphi: X_2 \rightarrow \overline{\mathrm{Aut}_\Lambda(P).C}$ of this assignment is indeed an isomorphism onto $\overline{\mathrm{Aut}_\Lambda(P).C}$.  For this purpose, we consider  the following affine open cover $(U_i)_{1 \le i \le 4}$ of $X_2$:  Namely, $U_1$ consists of the points with $z_0 = y_0 = 1$, that is, the following copy 
$$U_1 = \{ \left(1: z_1\right) \times \left(1: z_1: z_1^2: y_3\right)  \mid z_1, y_3 \in k\};$$
of affine $2$-space; $U_2$ consists of the points with $z_0 = y_3 = 1$, that is,
$$U_2 = \{ \left(1: z_1\right) \times \left(y_0: z_1y_0 : z_1^2 y_0: 1\right)\mid z_1, y_0 \in k \};$$
$U_3$ consists of the points with $z_1 = y_2 = 1$, that is,
$$U_3 = \{ \left(z_0:  1 \right) \times \left(z_0^2: z_0 :1 : y_3\right)\mid z_0, y_3 \in k \};$$
and $U_4$ consists of the points with $z_1 = y_3 = 1$, that is,
$$U_4 = \{ \left(z_0: 1 \right) \times \left(z_0^2 y_2: z_0 y_2 :y_2: 1\right)\mid z_0, y_2 \in k \}.$$

In describing the restrictions $\varphi_i$ of $\varphi$ to the affine charts $U_i$, we use the following convention:  If $\tau: \mathbb{A}_k^1 \,-\, \{a_1, \dots, a_r\}\,  \longrightarrow \,\overline{\mathrm{Aut}_\Lambda(P).C}$ is a smooth curve and $\overline{\tau}$ its unique extension to a curve defined on $\mathbb{P}_k^1$, we denote $\overline{\tau}(a_j)$ by $\lim_{t \rightarrow a_j} \tau(t)$.

The morphism $\varphi_1$ sends a point $(1: z_1) \times (1: z_1: z_1^2: y_3)$ of $U_1$ to $C(e_1 + z_1 \omega + y_3 \omega^2)$, and thus induces an isomorphism from $U_1$ onto $\mathrm{Aut}_\Lambda(P).C$.

The restriction $\varphi_2$ sends a point $(1:z_1) \times (y_0: z_1y_0 :z_1^2 y_0: 1) \in U_2$ to $\, C(e_1 +  z_1 \omega + \frac{1}{y_0} \omega^2)$ if $y_0 \ne 0$, and to 
$$\lim_{t \rightarrow 0} \;C\left(e_1 +  z_1 \omega + \frac{1}{t} \omega^2\right)$$
if $y_0 = 0$.  For $t \ne 0$ a short calculation shows
$$C\left(e_1 +  z_1 \omega + \frac{1}{t} \omega^2\right) = 
\Lambda \left(t\alpha + t\beta + t z_1\alpha\omega+t z_1\beta\omega +\alpha \omega^2\right)
 + \Lambda \left(\alpha \omega + z_1 \alpha \omega^2\right) +
\Lambda \left(\gamma \omega + z_1 \gamma \omega^2\right).$$
Thus
$$\lim_{t \rightarrow 0} \;C\left(e_1 +  z_1 \omega + \frac{1}{t} \omega^2\right) \ =  \ \Lambda \alpha \omega^2 + \Lambda \alpha \omega  + \Lambda\left(\gamma \omega + z_1 \gamma \omega^2\right).$$
As $z_1$ traces $k$, the latter points trace the orbit $\mathrm{Aut}_\Lambda(P).E$, where $E =  \Lambda \alpha \omega^2 + \Lambda \alpha \omega + \Lambda \gamma \omega$.  This is the only $1$-dimensional orbit in $\overline{\mathrm{Aut}_\Lambda(P).C}$; the corresponding degeneration $P/E$ of $M = P/C$ is depicted in the left-hand position of the second row of Figure \ref{fig:X2}.  

The restriction $\varphi_3$ sends any point $(z_0: 1 ) \times (z_0^2 : z_0: 1 :y_3) \in U_3$ to $C(e_1 +  \frac{1}{z_0} \omega  + \frac{y_3}{z_0^2} \omega^2)$ if $z_0 \ne 0$ and to
$$\lim_{t \rightarrow 0}\; C\left(e_1 +  \frac{1}{t} \omega  + \frac{y_3}{t^2} \omega^2\right)$$
if $z_0 = 0$.  For $t \ne 0$ a calculation similar to the one above shows 
$$\lim_{t \rightarrow 0} \;C\left(e_1 +  \frac{1}{t} \omega  + \frac{y_3}{t^2} \omega^2\right) \ =  \ \Lambda \left( (1 - y_3) \alpha \omega + \beta \omega\right) + \Lambda \alpha \omega^2 + \Lambda \gamma \omega^2.$$
When $y_3 = 1$ this gives the point 
$\Lambda \beta \omega + \Lambda \alpha \omega^2 + \Lambda \gamma \omega^2 $ 
which constitutes  the $0$-dimensional orbit corresponding to the degeneration that appears on the 
far right in the second row of Figure \ref{fig:X2}.  As $1- y_3 $ varies over $ k^*$ (i.e. $y_3$ varies over 
$k-\{1\}$) we have a $k^*$-family of degeneration of $M=P/C$ as shown in the central position of 
Figure \ref{fig:X2}.

Finally, we consider the restriction $\varphi_4$ of $\varphi$ to points
$Q= \left(z_0: 1 \right) \times \left(z_0^2 y_2: z_0 y_2 :y_2: 1\right) \in U_4$. 
If $z_0 \ne 0$ and $y_2 \ne 0$, 
then $\varphi_4(Q)=C\left(e_1 + \frac{1}{z_0} \omega + \frac{1}{z_0^2 y_2} \omega^2\right)$. For $z_0 = 0$ and $y_2 \ne 0$,
we find as above that 
$$\varphi_4(Q)=\lim_{t \rightarrow 0} \; C\left(e_1 +  \frac{1}{t} \omega  + \frac{1}{t^2y_2} \omega^2\right) \ =  \ \Lambda \left( \left(1 - \frac{1}{y_2}\right) \alpha \omega + \beta \omega\right) + \Lambda \alpha \omega^2 + \Lambda \gamma \omega^2,$$
yielding points already encountered in $\mathrm{Im}(\varphi_3)$. 

If $z_0 \ne 0$ and $y_2= 0$, one finds in a similar way that
$$\varphi_4(Q) = \lim_{t \rightarrow 0} \; C\left(e_1 + \frac{1}{z_0} \omega + \frac{1}{z_0^2 t} \omega^2\right)  = \Lambda \alpha\omega \ + \Lambda \left(\alpha \omega + \frac{1}{z_0} \alpha \omega^2\right) \ + \ \Lambda\left( \gamma \omega + \frac{1}{z_0} \gamma \omega^2\right).$$
The resulting points in the boundary of $\mathrm{Aut}_\Lambda(P).C$, as $z_0$ ranges over $k^*$, retrace the points in the orbit $\mathrm{Aut}_\Lambda(P).E \,-\, \{E\}$,  where $E \in \mathrm{Im}(\varphi_2)$ is as specified above.

If $z_0 = y_2 = 0$, on the other hand, then
$$\varphi_4(Q) = \lim_{t \rightarrow 0} \left( \lim_{s \rightarrow 0} \; C\left(e_1 + \frac{1}{t} \omega + \frac{1}{t^2 s} \omega^2\right)  \right) = \Lambda \alpha \omega + \Lambda \alpha \omega^2 + \Lambda \gamma \omega^2.$$
This value of $\varphi_4$ gives rise to the maximal top-stable degeneration that appears in the third row of Figure \ref{fig:X2}.

One checks that the maps $\varphi_i$ coincide on the overlaps of their domains, that each yields an isomorphism from $U_i$ to an open subvariety of $\overline{\mathrm{Aut}_\Lambda(P).C}$, and that their images cover $\overline{\mathrm{Aut}_\Lambda(P).C}$.
\end{Example}

\begin{Example}
\label{ex:blowup}
In this example, we retain the same algebra $\Lambda=kQ/I$ introduced in Example \ref{ex:X2}, as
well as the projective module $P$. But we use a different point $C$ in $\mathfrak{Grass}^{S_1}_8$.
We will show that in this case the minimal desingularization 
$\overline{\mathrm{Aut}_\Lambda(P).C}^\dagger$ 
is isomorphic to a blow-up $B$ of $X_2$ at a closed point such that both $\mathbb{P}^2_k$ 
and $X_2$ are relatively minimal models. Moreover, we will show that the orbit closure 
$\overline{\mathrm{Aut}_\Lambda(P).C}$ is obtained from 
$B$ by blowing down a curve of self-intersection $-2$, resulting in a singular projective surface.
In other words, we have the following diagram of birational morphisms
\begin{equation}
\label{eq:nicepic}
\xymatrix{
&&&&&\overline{\mathrm{Aut}_\Lambda(P) . C}^\dagger\ar[d]^f\ar[dll]_{\rho_1}\ar[rrd]^{\rho_3}\\
&&&X_1\ar[llld]_{\rho_2}&&\overline{\mathrm{Aut}_\Lambda(P) . C}&&X_2\\
\mathbb{P}^2_k}
\end{equation}
where each of $\rho_1$, $\rho_2$ and $\rho_3$ results from blowing down a curve of self-intersection $-1$ and the morphism
$\overline{\mathrm{Aut}_\Lambda(P) . C}^\dagger \xrightarrow{\;f} \overline{\mathrm{Aut}_\Lambda(P) . C}$
results from blowing down a curve of self-intersection $-2$.

Let $\Lambda=kQ/I$ be as in Example \ref{ex:X2}, let $P=\Lambda e_1$, and define
$$C=\Lambda(\alpha+\beta) +\Lambda\alpha\omega +\Lambda\gamma= 
k(\alpha+\beta)+k\alpha\omega +k\gamma \in\mathfrak{Grass}^{S_1}_8.$$ 
The module $M=P/C$ is displayed at the top of the Figure \ref{fig:blowup}.
Again, $\mathrm{Aut}_\Lambda(P) .C$ consists of the points 
\begin{eqnarray*}
C(e_1+t_1\omega+t_2\omega^2) &=& 
k\left(\alpha+\beta+t_1\alpha\omega+t_1\beta\omega+t_2\alpha\omega^2\right) +\nonumber\\
&& k\left(\alpha\omega+t_1\alpha\omega^2\right) + k\left(\gamma+t_1\gamma\omega + t_2\gamma\omega^2\right) ,
\end{eqnarray*}
for $(t_1,t_2)\in\mathbb{A}^2_k$.
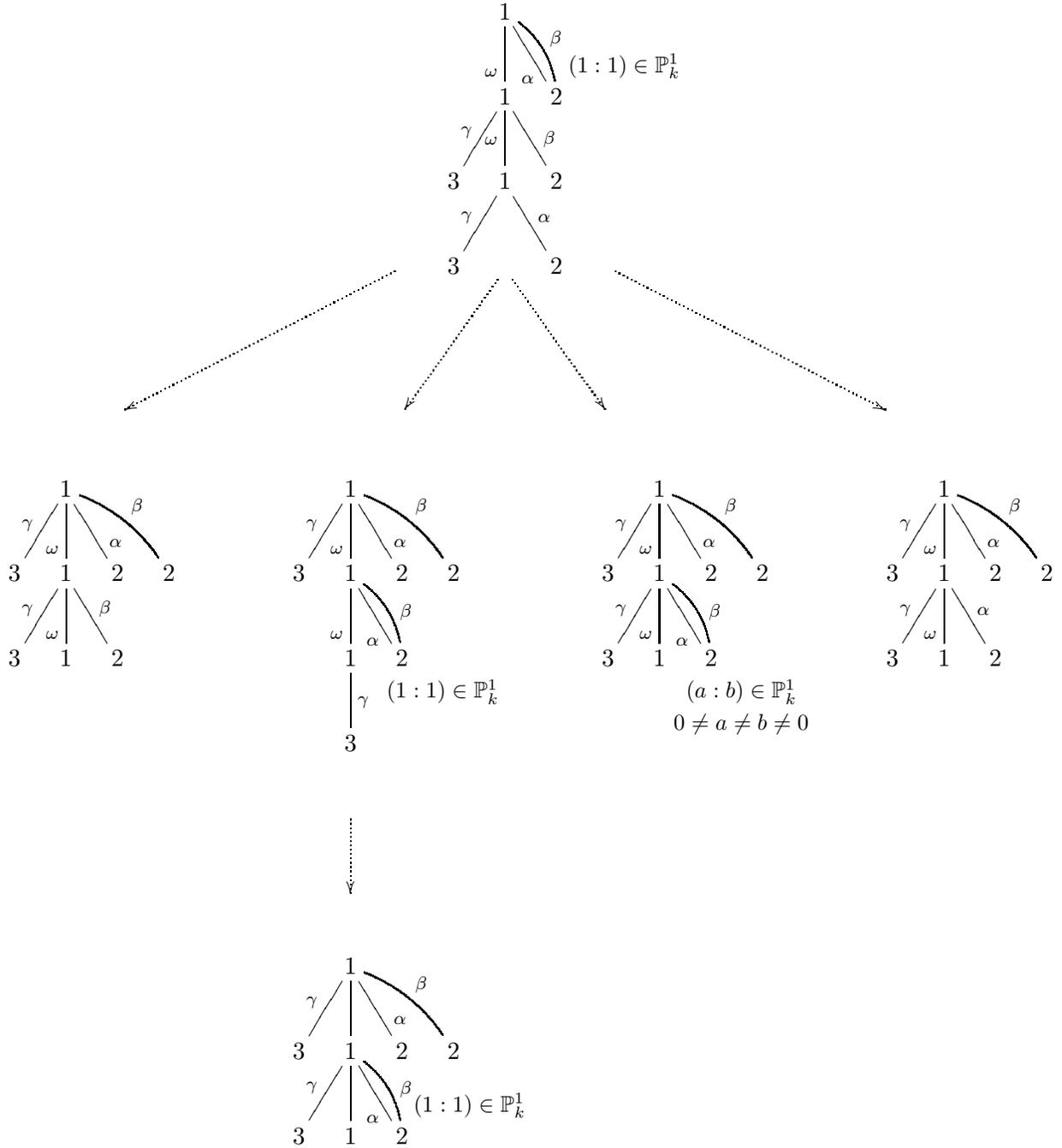
\begin{figure}[ht] \caption{\label{fig:blowup} The hierarchy of top-stable degenerations  in 
Example \ref{ex:blowup}.}
$$ \xymatrixrowsep{2.0pc}\xymatrixcolsep{.9pc}
\xymatrix{
 &&&&&&&&&&1 \edge[d]_(0.7){\omega} \edge[dr]_(0.7){\alpha} \edge@/^/[dr]^{\beta}  \\
 &&&&&&&&& &1 \edge[dl]_{\gamma} \edge[d]_{\omega} \edge[dr]^(0.6){\beta} &2 & 
 {\save+<1.5ex,3ex> \drop{\mbox{\small $(1:1) \in \mathbb{P}^1_k$}}\restore}  \\
&&& &&&&&&3&1 \edge[dl]_{\gamma} \edge[dr]^{\alpha} &2  \\
&&& &&&&&\dttdar[ddllllll] &3&\dttdar[ddll]\dttdar[ddrr]& 2&\dttdar[ddrrrrrr]\\
 &&& &&&&&&& &&&&&&&& &&&&&  \\  &&& &&&&&&& &&&&&&&& &&&&& \\
 &1 \edge[dl]_{\gamma} \edge[d]_(0.7){\omega} \edge[dr]^(0.7){\alpha} \edge@/^/[drr]^{\beta} &&&&&&1 \edge[dl]_{\gamma} \edge[d]_(0.7){\omega} \edge[dr]^(0.7){\alpha} \edge@/^/[drr]^{\beta} &&&&&&1 \edge[dl]_{\gamma} \edge[d]_(0.7){\omega} \edge[dr]^(0.7){\alpha} \edge@/^/[drr]^{\beta}   &&&&&&1 \edge[dl]_{\gamma} \edge[d]_(0.7){\omega} \edge[dr]^( 0.7){\alpha} \edge@/^/[drr]^{\beta} && \\
3 &1 \edge[dl]_{\gamma} \edge[d]_(0.7){\omega} \edge[dr]^{\beta} &2 &2 &&&3 &1  \edge[d]_(0.7){\omega} \edge[dr]_(0.7){\alpha} \edge@/^/[dr]^(0.6){\beta}  &2 &2 &&&3 &1 \edge[dl]_{\gamma} \edge[d]_(0.7){\omega} \edge[dr]_(0.7){\alpha} \edge@/^/[dr]^(0.6){\beta} &2 &2&&&3&1 \edge[dl]_{\gamma} \edge[d]_(0.7){\omega} \edge[dr]^{\alpha} &2&2  \\
3 &1 &2 &&&& &1\edge[d]^{\gamma} &2 &&&&3 &1 &2  & &&&3&1&2 \\
 &&&&&&&  {\save[0,0]+<0ex,-7ex> \drop{} \ar@{}^{}="BB" \restore}  {\save[3,0]+<0ex,7ex> \drop{} \ar@{}^{}="CC" \restore} \dttdar"BB";"CC" 3&&{\save+<-1ex,5ex> \drop{\mbox{\small $(1:1) \in \mathbb{P}_k^1$}} \restore} &&&&&&{\save+<-2ex,5ex> \drop{\mbox{\small $(a:b) \in \mathbb{P}_k^1$}} \restore} 
 {\save+<-2ex,2ex> \drop{\mbox{\small $0 \neq a\neq b \neq 0$}} \restore}  \\ \\  \\
 &&&&&&&1 \edge[dl]_{\gamma} \edge[d] \edge[dr]^(0.7){\alpha} \edge@/^/[drr]^{\beta}  \\
&&&&&&3 &1 \edge[dl]_{\gamma} \edge[d] \edge[dr]_(0.7){\alpha} \edge@/^/[dr]^(0.6){\beta} &2 &2  \\
&&&&&&3 &1 &2 &{\save+<1.5ex,3ex> \drop{\mbox{\small $(1:1) \in \mathbb{P}_k^1$}} \restore}
}$$
\end{figure}

We consider the blow-up $B$ of 
$$X_2 = \{ (z_0: z_1) \times (y_0: y_1: y_2: y_3)  \in \mathbb{P}^1_k \times \mathbb{P}^3_k \mid z_0 y_1 = z_1 y_0 \text{\ and \ } z_0 y_2 = z_1 y_1\}$$
at the closed point $(z_0:z_1) \times (y_0:y_1:y_2:y_3) = (0:1) \times (0:0:1:0)$. In other words, 
\begin{eqnarray*}
B&=&\{ (z_0: z_1) \times (y_0: y_1: y_2: y_3)\times (x_0:x_1)  \in \mathbb{P}^1_k \times \mathbb{P}^3_k 
\times \mathbb{P}^1_k \mid \\
&& \qquad\qquad\qquad\qquad\qquad z_0 y_1 = z_1 y_0, z_0 y_2 = z_1 y_1 
\mbox{ and } z_0x_1=y_3x_0\}.
\end{eqnarray*}
We consider the following affine cover of $B$, consisting of 5 affine 2-spaces:
\begin{eqnarray*}
\tilde{U}_1 &=& \{ \left(1: z_1\right) \times \left(1: z_1: z_1^2: y_3\right)\times (1:y_3)  \mid z_1, y_3 \in k\};\\
\tilde{U}_2 &=& \{ \left(1: z_1\right) \times \left(y_0: z_1y_0 : z_1^2 y_0: 1\right)\times (1:1)\mid z_1, y_0 \in k \};\\
\tilde{U}_{31} &=& \{ \left(z_0:  1 \right) \times \left(z_0^2: z_0 :1 : \xi z_0\right)\times (1:\xi)\mid z_0, \xi \in k \};\\
\tilde{U}_{32} &=& \{ \left(\zeta y_3:  1 \right) \times \left(\zeta^2y_3^2: \zeta y_3:1 : y_3\right)\times (\zeta:1)\mid \zeta, y_3 \in k \};\\
\tilde{U}_4 &=& \{ \left(z_0: 1 \right) \times \left(z_0^2 y_2: z_0 y_2 :y_2: 1\right)\times (z_0:1)\mid z_0, y_2 \in k \}.
\end{eqnarray*}
We now define a morphism $\varphi: B \to \overline{\mathrm{Aut}_\Lambda(P).C}$ by defining
the restrictions of $\varphi$ to these 5 affine patches as follows.

The restriction $\varphi_1$ of $\varphi$ to $\tilde{U}_1$
sends a point $(1: z_1) \times (1: z_1: z_1^2: y_3)\times (1:y_3)$ of $\tilde{U}_1$ to $C(e_1 + z_1 \omega + y_3 \omega^2)$, and thus induces an isomorphism from $\tilde{U}_1$ onto $\mathrm{Aut}_\Lambda(P).C$.

The restriction $\varphi_2$ sends a point $(1:z_1) \times (y_0: z_1y_0 :z_1^2 y_0: 1)\times (1:1)$
of $\tilde{U}_2$ to $\, C\left(e_1 +  z_1 \omega + \frac{1}{y_0} \omega^2\right)$ if $y_0 \ne 0$, and to 
\begin{eqnarray*}
\lim_{t \rightarrow 0} \;C\left(e_1 +  z_1 \omega + \frac{1}{t} \omega^2\right) &=&
\Lambda \alpha \omega^2 + \Lambda \alpha \omega  + \Lambda \gamma\omega^2
\end{eqnarray*}
if $y_0=0$.
The corresponding degeneration of $M = P/C$ is depicted in the left-hand position of the second row of Figure \ref{fig:blowup}.  

The restriction $\varphi_{31}$ sends any point $(z_0: 1 ) \times (z_0^2 : z_0: 1 :\xi z_0)\times (1:\xi) \in \tilde{U}_{31}$ to $\, C\left(e_1 +  \frac{1}{z_0} \omega  + \frac{\xi}{z_0} \omega^2\right)$ if $z_0 \ne 0$ and to
\begin{eqnarray*}
\lim_{t \rightarrow 0}\; C\left(e_1 +  \frac{1}{t} \omega  + \frac{\xi}{t} \omega^2\right)&=&
\Lambda (\alpha\omega+\beta\omega) + \Lambda \alpha\omega^2 + \Lambda (\gamma\omega
+ \xi \gamma\omega^2)
\end{eqnarray*}
if $z_0 = 0$.  
As $\xi$ traces $k$, the latter points trace the orbit $\mathrm{Aut}_\Lambda(P).D$, where 
$D =  \Lambda (\alpha\omega+\beta\omega) + \Lambda \alpha\omega^2 + \Lambda \gamma\omega$. 
 This is the only $1$-dimensional orbit in $\overline{\mathrm{Aut}_\Lambda(P).C}$; the corresponding degeneration $P/D$ of $M = P/C$ is depicted in the second-to-left position of the second row of Figure \ref{fig:blowup}.  
 
Next, we consider the restriction $\varphi_{32}$ of $\varphi$ to points
$R= (\zeta y_3: 1 ) \times (\zeta^2 y_3^2 : \zeta y_3: 1 :y_3)\times (\zeta:1) \in \tilde{U}_{32}$. If $\zeta\neq 0$ and $y_3\neq 0$,
then $\varphi_{32}(R)=C(e_1 +  \frac{1}{\zeta y_3} \omega  + \frac{1}{\zeta^2 y_3} \omega^2)$.
For $\zeta=0$ and $y_3\neq 0$, we get that $\varphi_{32}(R)$ is equal to the limit
\begin{eqnarray*}
\lim_{t \rightarrow 0}\;  C\left(e_1 +  \frac{1}{t y_3} \omega  + \frac{1}{t^2 y_3} \omega^2\right)
&=& \Lambda \left(\beta\omega + (1-y_3) \alpha\omega\right) + 
\Lambda \alpha\omega^2 + \Lambda \gamma\omega^2.
\end{eqnarray*}
When $y_3=1$ this gives the point 
$\Lambda \beta\omega+\Lambda \alpha\omega^2 + \Lambda \gamma\omega^2$ which constitutes
the 0-dimensional orbit corresponding to the degeneration that appears on the far right in the second
row of Figure \ref{fig:blowup}.  
As $1-y_3$ varies over $k^*$ (i.e. $y_3$ varies over $k-\{1\}$) we obtain a $k^*$-family of 
degenerations of $M=P/C$ as shown in the second-to-right position of the second row of Figure \ref{fig:blowup}.  

If $\zeta\neq 0$ and $y_3=0$, one finds in a similar way that
$$\varphi_{32}(R)\ = \
\lim_{t \rightarrow 0}\;  C\left(e_1 +  \frac{1}{\zeta t} \omega  + \frac{1}{\zeta^2 t} \omega^2\right)
= \Lambda (\alpha\omega+\beta\omega) + \Lambda \alpha\omega^2 + \Lambda \left(\gamma\omega
+\frac{1}{\zeta} \gamma\omega^2\right).$$
The resulting points in the boundary of $\mathrm{Aut}_\Lambda(P).C$, as $\zeta$ ranges over
$k^*$, retrace the points in the orbit $\mathrm{Aut}_\Lambda(P).D$, where $D\in\mathrm{Im}(\varphi_{31})$ is specified as above.

If $\zeta=y_3=0$, on the other hand, then
$$\varphi_{32}(R) \ = \ 
\lim_{t \rightarrow 0} \left( \lim_{s \rightarrow 0} \; C\left(e_1 + \frac{1}{ts} \omega + \frac{1}{t^2 s} \omega^2\right)  \right) = \Lambda (\alpha \omega + \beta\omega)
+ \Lambda \alpha \omega^2 + \Lambda \gamma \omega^2.$$
This value of $\varphi_{32}$ gives rise to the maximal top-stable degeneration that appears in the third row of Figure \ref{fig:blowup}.

Finally, we consider the restriction $\varphi_4$ of  $\varphi$ to points
$Q= \left(z_0: 1 \right) \times \left(z_0^2 y_2: z_0 y_2 :y_2: 1\right)\times (z_0:1) \in \tilde{U}_4$. 
If $z_0 \ne 0$ and $y_2 \ne 0$, then $\varphi_4(Q)=C\left(e_1 + \frac{1}{z_0} \omega + \frac{1}{z_0^2 y_2} \omega^2\right)$. For $z_0 = 0$ and $y_2 \ne 0$, we get that
$\varphi_4(Q)$ is equal to the limit
\begin{eqnarray*}
\lim_{t \rightarrow 0} \; C\left(e_1 +  \frac{1}{t} \omega  + \frac{1}{t^2 y_2} \omega^2\right) &=&
\Lambda\left(\beta \omega + \left(1 - \frac{1}{y_2}\right) \alpha \omega\right) + \Lambda\alpha\omega^2 + \Lambda \gamma\omega^2.
\end{eqnarray*}
Thus, we obtain points already encountered in $\mathrm{Im}(\varphi_{32})$. 

If $z_0 \ne 0$ and $y_2= 0$, one finds in a similar way that
$$\varphi_4(Q) = \lim_{t \rightarrow 0} \; C\left(e_1 + \frac{1}{z_0} \omega + \frac{1}{z_0^2 t} \omega^2\right)  = \Lambda \alpha\omega \ + \Lambda \left(\alpha \omega + \frac{1}{z_0} \alpha \omega^2\right) \ + \ \Lambda \gamma \omega^2,$$
yielding the degeneration already encountered in $\mathrm{Im}(\varphi_2)$. 
If $z_0 = y_2 = 0$, then we again obtain that
$$\varphi_4(Q) = \lim_{t \rightarrow 0} \left( \lim_{s \rightarrow 0} \; C\left(e_1 + \frac{1}{t} \omega + \frac{1}{t^2 s} \omega^2\right)  \right) = \Lambda \alpha \omega + \Lambda \alpha \omega^2 + \Lambda \gamma \omega^2.$$

One checks that the maps $\varphi_i$ for $i\in\{1,2,31,32,4\}$ coincide on the overlaps of their domains,
and that their images cover $\overline{\mathrm{Aut}_\Lambda(P).C}$. 
Moreover, for $i\in\{1,31,32\}$, $\varphi_i$ yields an 
isomorphism from $U_i$ to an open subvariety of $\overline{\mathrm{Aut}_\Lambda(P).C}$.
On $U_2$, however, $\varphi_2$ sends the curve $y_0=0$, corresponding to the points
$(1:z_1)\times (0:0:0:1)$ for $z_1\in k$, to the single point $\Lambda\alpha\omega^2+
\Lambda \alpha\omega + \Lambda\gamma\omega^2$.
Similarly, on $U_4$, $\varphi_4$ sends the curve $y_2=0$, corresponding to the points
$(z_0:1)\times (0:0:0:1)$ for $z_0\in k$, again to the same point $\Lambda\alpha\omega^2+
\Lambda \alpha\omega + \Lambda\gamma\omega^2$. The curves $y_0=0$ on $U_2$, respectively,
$y_2=0$ on $U_4$, define the same projective curve on the blow-up $B$. We need to blow
down this curve on $B$ to a point.
Our explicit description of the maps $\varphi_i$ shows that we do not need to blow down any
further curves on $B$, but that this blow-down is isomorphic to 
the orbit closure $\overline{\mathrm{Aut}_\Lambda(P).C}$. 
We also see that the point on $X_2$ we blew up to construct $B$ does not lie on the curve we
need to blow down.
Computing the self-intersections of all the curves on $B$ corresponding to boundary points of
$\overline{\mathrm{Aut}_\Lambda(P).C}$, it follows that this curve has
self-intersection $-2$. Therefore, $\overline{\mathrm{Aut}_\Lambda(P).C}$ is a singular projective 
surface.

The blow-up $B'$ of
$X_1=\{ (z_0: z_1) \times (u_0: u_1: u_2)  \in \mathbb{P}^1_k \times \mathbb{P}^2_k \mid z_0 u_1 = z_1 u_0\}$
at the closed point $(z_0: z_1) \times (u_0: u_1: u_2) = (0:1)\times (0:0:1)$ can be described as
$$B'=\{ (z_0: z_1) \times (u_0: u_1: u_2) \times (t_0:t_1) \in \mathbb{P}^1_k \times \mathbb{P}^2_k \times \mathbb{P}^1_k \mid z_0 u_1 = z_1 u_0 \mbox{ and } z_0t_1 = u_1t_0\}.$$
Using an appropriate affine open cover, consisting of 5 affine 2-spaces,
it is straightforward to check that $B$ and $B'$ are isomorphic as schemes, giving the
diagram $(\ref{eq:nicepic})$.
\end{Example}

\begin{Example}
\label{ex:blowupsmooth}
In our last example, $\overline{\mathrm{Aut}_\Lambda(P) . C}$ is a smooth rational projective surface that
fails to be relatively minimal. More precisely, $\overline{\mathrm{Aut}_\Lambda(P).C}$ 
is isomorphic to the blow-up $B_0$ of $\mathbb{P}^1_k\times\mathbb{P}^1_k$ at $\infty\times \infty$
such that both $\mathbb{P}^2_k$ and $\mathbb{P}^1_k\times\mathbb{P}^1_k$ are relatively 
minimal models. 

Let $\Lambda=kQ/I$, where
\begin{eqnarray*}
Q &=& \quad {\xymatrix{
3&&\ar@<0.8ex>[ll]^{\delta}  \ar@<-0.8ex>[ll]_{\gamma} 1 
\ar@'{@+{[0,0]+(4,4)} @+{[0,0]+(0,15)}
@+{[0,0]+(-4,4)}}_(.4){\omega_1}
\ar@'{@+{[0,0]-(4,4)} @+{[0,0]-(0,15)}
@+{[0,0]+(4,-4)}}_(.6){\omega_2}
\ar@<0.8ex>[rr]^{\alpha} \ar@<-0.8ex>[rr]_{\beta} &&2}}
\qquad\mbox{ and}\\
I &=& \quad\langle\omega_1^2,\omega_1\omega_2,\omega_2\omega_1,\omega_2^2,\alpha\omega_2,
\beta\omega_1,\gamma\omega_2,\delta\omega_1\rangle\;.
\end{eqnarray*}
As before, let $P=\Lambda e_1$ and
$$C=\Lambda \alpha + \Lambda \beta + \Lambda \gamma +  \Lambda \delta + \Lambda (\alpha \omega_1 + \beta \omega_2) = 
k \alpha + k \beta + k \gamma +  k \delta + k (\alpha \omega_1 + \beta \omega_2)
\in \mathfrak{Grass}^{S_1}_{6}.$$
The module $M=P/C$ is displayed at the top of Figure \ref{fig:blowupsmooth}.
Again, $\mathrm{Aut}_\Lambda(P) .C$ consists of the points 
\begin{eqnarray*}
C(e_1+t_1\omega_1+t_2\omega_2) &=& 
k(\alpha+\beta+t_1\alpha\omega_1)+k(\beta+t_2\beta\omega_2)
+k(\gamma+t_1\gamma\omega_1) + \nonumber\\
&& k(\delta+t_2\delta\omega_2) + k(\alpha\omega_1+\beta\omega_2)  ,
\end{eqnarray*}
for $(t_1,t_2)\in\mathbb{A}^2_k$.
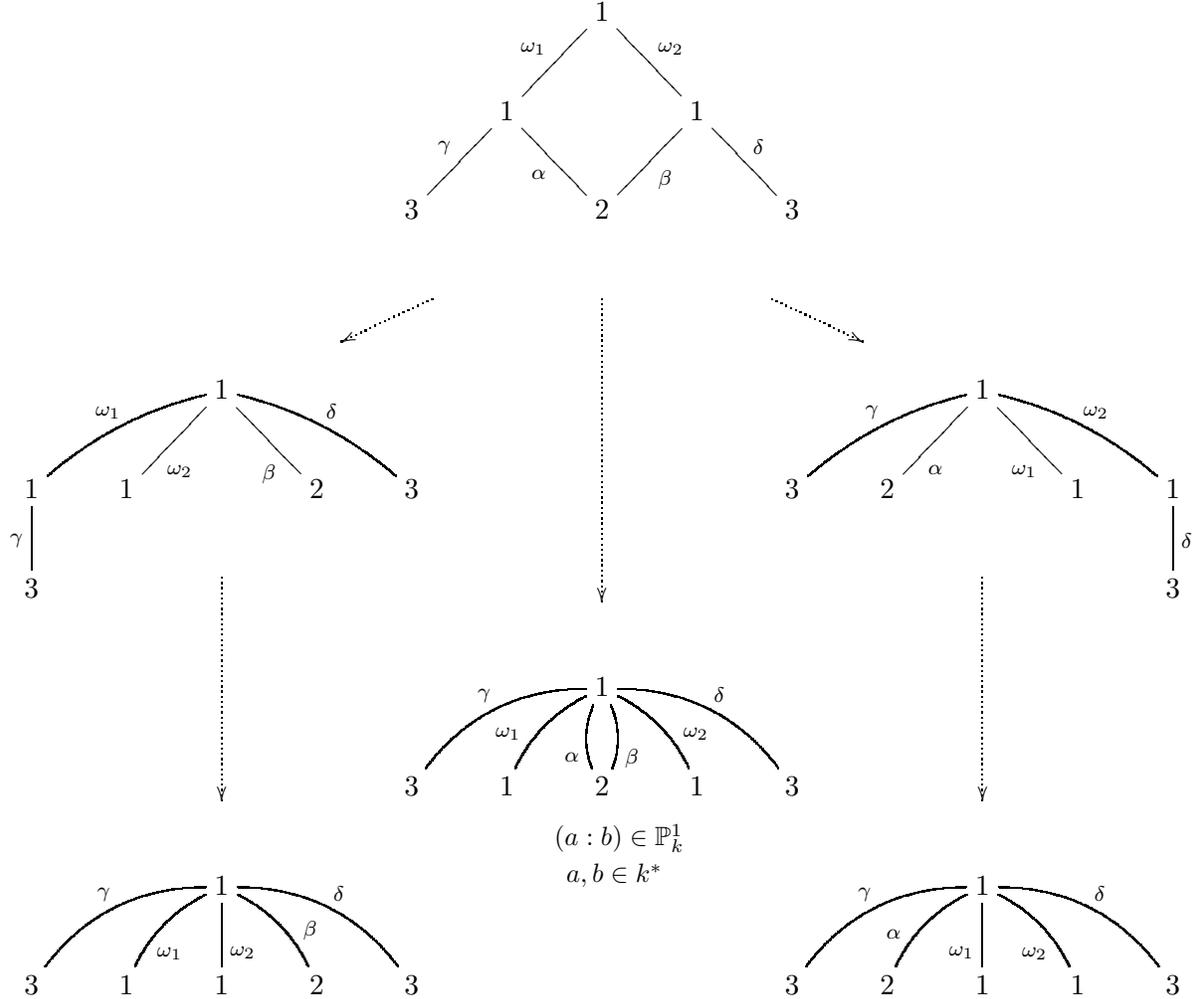
\begin{figure}[ht] \caption{\label{fig:blowupsmooth} The hierarchy of top-stable degenerations  in 
Example \ref{ex:blowupsmooth}.}
$$\xymatrix{
 &&&&&&1 \edge[dl]_{\omega_1} \edge[dr]^{\omega_2}  \\
 &&&&&1 \edge[dl]_{\gamma} \edge[dr]_{\alpha} &&1 \edge[dl]^{\beta} \edge[dr]^{\delta}  \\
 &&&&3 &{\save[0,0]+<-6ex,-7ex> \drop{} \ar@{}^{}="AA" \restore} &2 {\save[0,0]+<0ex,-7ex> \drop{} \ar@{}^{}="BB" \restore} &{\save[0,0]+<6ex,-7ex> \drop{} \ar@{}^{}="CC" \restore} &3  \\  \\  
 &&1 \edge@/_/[dll]_{\omega_1} \edge[dl]^(0.66){\omega_2} \edge[dr]_(0.66){\beta} \edge@/^/[drr]^{\delta}  &{\save[0,0]+<2ex,4ex> \drop{} \ar@{}^{}="AAA" \restore} \dttdar"AA";"AAA" &&& &&&{\save[0,0]+<-2ex,4ex> \drop{} \ar@{}^{}="CCC" \restore} \dttdar"CC";"CCC" &1 \edge@/_/[dll]_{\gamma} \edge[dl]^(0.66){\alpha} \edge[dr]_(0.66){\omega_1} \edge@/^/[drr]^{\omega_2}  \\
1 \edge[d]_{\gamma} &1 & {\save[0,0]+<0ex,-7ex> \drop{} \ar@{}^{}="DD" \restore} &2 &3 &&&&3 &2 & {\save[0,0]+<0ex,-7ex> \drop{} \ar@{}^{}="EE" \restore} &1 &1 \edge[d]^{\delta}  \\
3 &&&&&& &&&&&&3  \\  
 &&&&&&1 \edge@/_1pc/[dll]_{\gamma} \edge@/_/[dl]_(0.66){\omega_1} \edge@/_/[d]_(0.66){\alpha} \edge@/^/[d]^(0.66){\beta} \edge@/^/[dr]^(0.66){\omega_2} \edge@/^1pc/[drr]^{\delta}   {\save[0,0]+<0ex,7ex> \drop{} \ar@{}^{}="BBB" \restore} \dttdar"BB";"BBB"  \\
 &&&&3 &1 &2 {\save+<1ex,-4ex> \drop{\mbox{ \small $(a:b) \in \mathbb{P}_k^1$}} \restore} {\save+<1ex,-7ex> \drop{\mbox{\small $a,b \in k^*$}} \restore} &1 &3  \\
 &&1 \edge@/_1pc/[dll]_{\gamma} \edge@/_/[dl]^(0.66){\omega_1} \edge[d]^(0.66){\omega_2} \edge@/^/[dr]^(0.66){\beta} \edge@/^1pc/[drr]^{\delta} {\save[0,0]+<0ex,7ex> \drop{} \ar@{}^{}="DDD" \restore} \dttdar"DD";"DDD" &&&& &&&&1 \edge@/_1pc/[dll]_{\gamma} \edge@/_/[dl]_(0.66){\alpha} \edge[d]_(0.66){\omega_1} \edge@/^/[dr]_(0.66){\omega_2} \edge@/^1pc/[drr]^{\delta} {\save[0,0]+<0ex,7ex> \drop{} \ar@{}^{}="EEE" \restore} \dttdar"EE";"EEE"  \\
3 &1 &1 &2 &3 &&&&3 &2 &1 &1 &3
 }$$
 \end{figure}

Consider the blow-up $B_0$ of $\mathbb{P}^1_k\times \mathbb{P}^1_k$ at the closed point
$(z_0:z_1)\times (y_0:y_1)=(0:1)\times (0:1)$. More 
precisely,
$$B_0=\{(z_0:z_1)\times (y_0:y_1)\times (x_0:x_1)\in \mathbb{P}^1_k\times \mathbb{P}^1_k\times
\mathbb{P}^1_k \mid z_0x_1 = y_0x_0\}.$$
Let $O_1, O_2, O_3, O_4$ and $ O_5$ be the affine patches where $z_0 y_0, z_0 y_1, z_1 y_0, z_1 y_1 x_0$ and $
z_1 y_1 x_1$, respectively, are not $0$. Using these affine patches, one uses calculations similar
to those of Example \ref{ex:blowup} to produce  an isomorphism $\varphi:B_0\to \overline{\mathrm{Aut}_\Lambda(P) . C}$
with the following properties.  

A point $(z_0:z_1)\times (y_0:y_1)\times (x_0:x_1)$ of $B_0$ with $z_0\neq 0$ and $y_0\neq 0$
is sent to the point
$C\left( e_1 + \frac{z_1}{z_0} \omega_1 + \frac{y_1}{y_0}\omega_2\right) \in 
\mathrm{Aut}_\Lambda(P) . C\,.$

The restriction of $\varphi$ to $O_2$ sends a point $(1:z_1)\times (y_0:1)\times (1:y_0)$
to $\, C\left(e_1 +  z_1 \omega_1 + \frac{1}{y_0} \omega_2\right)$ if $y_0 \ne 0$, and to 
$\Lambda \alpha + \Lambda\alpha\omega_1 + \Lambda \beta\omega_2 + 
\Lambda (\gamma + z_1 \gamma\omega_1) + \Lambda\delta\omega_2$
if $y_0=0$.
As $z_1$ traces $k$, the latter points trace the orbit $\mathrm{Aut}_\Lambda(P).E_1$, where 
$E_1 = \Lambda \alpha + \Lambda \gamma + \Lambda\alpha\omega_1 + \Lambda \beta\omega_2 + 
\Lambda\delta\omega_2$. The corresponding degeneration $P/E_1$ of $M = P/C$ is depicted in the 
left position of the second row of Figure \ref{fig:blowupsmooth}.  

Analogously, the restriction of $\varphi$ to $O_3$ sends a point $(z_0:1)\times (1:y_1)\times (z_0:1)$ of $O_3$ 
to $\, C\left(e_1 +  \frac{1}{z_0} \omega_1  + y_1\omega_2\right)$ if $z_0 \ne 0$ and to
$ \Lambda \alpha\omega_1 + \Lambda\beta + \Lambda\beta\omega_2 +
\Lambda \gamma\omega_1 + \Lambda(\delta + y_1\delta\omega_2) $
if $z_0=0$.
As $y_1$ traces $k$, the latter points trace the orbit $\mathrm{Aut}_\Lambda(P).E_2$, where 
$E_2 = \Lambda \alpha\omega_1 + \Lambda\beta + \Lambda\beta\omega_2 +
\Lambda \gamma\omega_1 + \Lambda\delta$. The corresponding degeneration $P/E_2$ of $M = P/C$ is 
depicted in the right position of the second row of Figure \ref{fig:blowupsmooth}.  

The restriction of $\varphi$  to $O_4$  has the following image on $R=(z_0:1)\times (\mu z_0:1)\times (1:\mu)$. If $z_0\neq 0$ and $\mu\neq 0$, then $\varphi_4(R)=
C\left(e_1 +  \frac{1}{z_0} \omega_1  + \frac{1}{\mu z_0}\omega_2\right)$. If $z_0=0$ and $\mu\neq 0$,
$
\varphi_4(R)  = \Lambda (\alpha + \mu\beta) +\Lambda \alpha\omega_1 + \Lambda \beta\omega_2 + 
\Lambda \gamma\omega_1 + \Lambda \delta \omega_2$.
As $\mu$ ranges over $k^*$, we obtain the degenerations in the third row of Figure \ref{fig:blowupsmooth}
corrsponding to the points $(1:\mu)\in\mathbb{P}^1_k$ for $\mu\in k^*$.
If $z_0\neq 0$ and $\mu = 0$, 
$\varphi_4(R) = \Lambda \alpha + \Lambda\alpha\omega_1 + \Lambda \beta\omega_2 + 
\Lambda \left(\gamma + \frac{1}{z_0} \gamma\omega_1\right) + \Lambda\delta\omega_2$;
these are degenerations already encountered in the image of $\varphi_2$.
If $z_0=0=\mu$, then
$\varphi_4(R) = \Lambda \alpha + \Lambda\alpha\omega_1 + \Lambda \beta\omega_2 + 
\Lambda \gamma\omega_1 + \Lambda\delta\omega_2.$
This is the 0-dimensional orbit corresponding to the degeneration in the left position of the fourth row
of Figure \ref{fig:blowupsmooth}.

The restriction of $\varphi$ to points $Q=(\nu y_0:1)\times (y_0:1)\times (\nu:1)\in O_5$ yields the following additional degenerations.
For $\nu\neq 0$ and $y_0= 0$,
$\varphi_5(Q)=\Lambda (\nu\alpha + \beta) +\Lambda \alpha\omega_1 + \Lambda \beta\omega_2 + 
\Lambda \gamma\omega_1 + \Lambda \delta \omega_2.$
As $\nu$ ranges over $k^*$, we obtain the degenerations in the third row of Figure \ref{fig:blowupsmooth}
corresponding to the points $(\nu:1)\in\mathbb{P}^1_k$ for $\mu\in k^*$.
On the other hand, if $\nu=0=y_0$, then
$\varphi_5(Q) =\Lambda\alpha\omega_1 + \Lambda \beta +  \Lambda \beta\omega_2 + 
\Lambda \gamma\omega_1 + \Lambda\delta\omega_2.$
This is the 0-dimensional orbit corresponding to the degeneration in the right position of the fourth row
of Figure \ref{fig:blowupsmooth}.

The blow-up $B_0'$ of
$X_1=\{ (z_0: z_1) \times (u_0: u_1: u_2)  \in \mathbb{P}^1_k \times \mathbb{P}^2_k \mid z_0 u_1 = z_1 u_0\}$
at the closed point $(z_0,u_2)=(0,0)$ can be described as
$$B_0'=\{ (z_0: z_1) \times (u_0: u_1: u_2) \times (t_0:t_1) \in \mathbb{P}^1_k \times \mathbb{P}^2_k \times \mathbb{P}^1_k \mid z_0 u_1 = z_1 u_0 \mbox{ and } z_0t_1 = u_2t_0\}.$$
Using an appropriate affine open cover of $B_0'$, consisting of 5 affine 2-spaces, we see that
$B_0$ and $B_0'$ are isomorphic as schemes.
\end{Example}

\begin{ConQuestions}
Let $T$ be a simple module, $P$ its projective cover, and $C \subseteq JP$ such that 
$\dim \mathrm{Aut}_\Lambda(P).C = 2$.
\vspace{1ex}
\begin{itemize}
\item[(1)] Is there a uniform upper bound (not depending on $k$ and $\mathrm{dim}_k C$) on the 
positive integers $n$ with the property that $\overline{\mathrm{Aut}_\Lambda(P).C}^{\dagger}$ has a 
relatively minimal model among $\mathbb{P}_k^2, X_0, X_2, \ldots, X_n$? 
\vspace{1ex}
\item[(2)] In Example \ref{ex:blowupsmooth}, the boundary of $\overline{\mathrm{Aut}_\Lambda(P).C}$ 
has three irreducible components; in all other examples, the number of irreducible components is 
$1$ or $2$.  Can more than three components be realized in $\overline{\mathrm{Aut}_\Lambda(P).C} 
- \mathrm{Aut}_\Lambda(P).C$ under our side conditions on $T$ and $C$?
\vspace{1ex}
\item[(3)]  How does the geometric structure of the surface $\overline{\mathrm{Aut}_\Lambda(P).C}$
(resp. of $\overline{\mathrm{Aut}_\Lambda(P).C}^{\dagger}$) pertain to degeneration-theoretic 
information about the module $P/C$?
\end{itemize}
\end{ConQuestions}

\bibliographystyle{plain}

\end{document}